\documentclass[a4paper,12pt]{amsart}

\usepackage{fullpage}

\usepackage{graphicx}
\usepackage[colorlinks, linkcolor= blue, citecolor= red]{hyperref}
\usepackage{pdfsync}
\usepackage{color}
\usepackage{lmodern}

\usepackage[T1]{fontenc}
\usepackage[utf8]{inputenc}

\usepackage[all]{xy}

\usepackage{amsmath, amssymb, amsfonts, amscd, amsthm, mathrsfs}

\newcommand{\cupp}{\cup}
\newcommand{\inv}{\mathrm{inv}}

\renewcommand{\phi}{\varphi}

\newcommand{\z}{\mathbb{Z}}

\newcommand{\q}{\mathbb{Q}}
\newcommand{\f}{\mathbb{F}}

\newcommand{\U}{\mathbb{U}}
\newcommand{\uf}{\U_5(\f_2)}
\newcommand{\gal}{\operatorname{Gal}}
\newcommand{\cl}[1]{[#1]}
\newcommand{\li}[1]{{#1}^\times \!/ {#1}^{\times 2} }
\newcommand{\co}{\operatorname{Coind}}
\newcommand{\ev}{\operatorname{ev}}
\newcommand{\Hom}{\operatorname{Hom}}

\newcommand{\h}{\operatorname{H}}

\newcommand{\cent}{\mathscr{C}}
\newcommand{\E}{\mathscr{E}}
\newcommand{\spec}{\operatorname{Spec}}

\renewcommand{\tilde}{\widetilde}
\renewcommand{\bar}{\overline}
\newcommand{\bbar}[1]{\bar{\bar{#1}}}

\newcommand{\pierre}[1]{}

\newcommand{\bark}{\mkern2mu\overline{\mkern-2mu k\mkern-1.5mu}\mkern1.5mu}
\newcommand{\Gm}{\mathbf{G}_\mathrm{m}}

\newcommand{\Br}{\mathrm{Br}}
\newcommand{\Spec}{\mathrm{Spec}}

\renewcommand{\P}{{\mathbf P}}

\newcommand{\isoto}{\myxrightarrow{\,\sim\,}}

\makeatletter
\def\myrightarrow{{\setbox\z@\hbox{$\rightarrow$}\dimen0\ht\z@\multiply\dimen0 6\divide\dimen0 10\ht\z@\dimen0\box\z@}}
\def\myrightarrowfill@{\arrowfill@\relbar\relbar\myrightarrow}
\newcommand{\myxrightarrow}[2][]{\ext@arrow 0359\myrightarrowfill@{#1}{#2}}
\def\myleftarrow{{\setbox\z@\hbox{$\leftarrow$}\dimen0\ht\z@\multiply\dimen0 6\divide\dimen0 10\ht\z@\dimen0\box\z@}}
\def\myleftarrowfill@{\arrowfill@\myleftarrow\relbar\relbar}
\newcommand{\myxleftarrow}[2][]{\ext@arrow 3095\myleftarrowfill@{#1}{#2}}
\makeatother
\newcommand{\sO}{{\mathscr O}}
\newcommand{\A}{{\mathbf A}}
\renewcommand{\P}{{\mathbf P}}
\newcommand{\Q}{{\mathbf Q}}
\newcommand{\Z}{{\mathbf Z}}
\newcommand{\nr}{\mathrm{nr}}
\newcommand{\et}{\text{\'et}}
\renewcommand{\phi}{\varphi}
\renewcommand{\emptyset}{\varnothing}

\newcommand{\Gal}{\operatorname{Gal}}
\newcommand{\F}{\mathbb{F}}

\renewcommand{\H}{\operatorname{H}}

\newtheoremstyle{pedro}{}{}{\itshape}{}{\sc}{~--}{ }{\thmname{#1}\thmnumber{ #2}\thmnote{ (#3)}}

\newtheoremstyle{pedrodef}{}{}{}{}{\sc}{~--}{ }{\thmname{#1}\thmnumber{ #2}\thmnote{ (#3)}}

\theoremstyle{pedro}
\newtheorem{lem}{Lemma}[section]

\newtheorem{thm}[lem]{Theorem}

\newtheorem{prop}[lem]{Proposition}

\newtheorem{maintheorem}{Theorem}

\theoremstyle{remark}

\newtheorem{rmk}[lem]{Remark}
\newtheorem*{rmk2}{Remark}

\theoremstyle{pedrodef}

\newtheorem{ex}[lem]{Example}

\title{Four-fold Massey products in Galois Cohomology}

\author{Pierre Guillot}
\address{
Pierre Guillot \\
Universit\'{e} de Strasbourg \& CNRS\\
Institut de Recherche Math\'{e}matique Avanc\'{e}e\\
UMR 7501, F-67000 Strasbourg, France}
\email{guillot@math.unistra.fr}

\author{J\'an Min\'a\v{c}}
\thanks{JM is partially supported by the Natural Sciences and Engineering Research Council of Canada (NSERC) grant R0370A01}
\address{
J\'an Min\'a\v{c} \\
Department of Mathematics\\
Western University\\
London, Ontario, N6A 5B7\\
Canada}
\email{minac@uwo.ca}

\author{Adam Topaz \\ With an Appendix by Olivier Wittenberg}
\thanks{AT was partially supported by EPSRC programme grant EP/M024830/1 Symmetries and Correspondences}
\address{
Adam Topaz \\
Mathematical Institute \\
University of Oxford \\
Andrew Wiles Building \\
Radcliffe Observatory Quarter \\
Woodstock Road \\
Oxford OX2 6GG \\
United Kingdom }
\email{topaz@maths.ox.ac.uk}

\keywords{Massey Products, Galois Cohomology, Splitting Variety, Group Cohomology.}
\subjclass[2010]{55S30, 12G05}

\let\oldtocsection=\tocsection
\let\oldtocsubsection=\tocsubsection
\let\oldtocsubsubsection=\tocsubsubsection

\renewcommand{\tocsection}[2]{\hspace{0em}\oldtocsection{#1}{#2}}
\renewcommand{\tocsubsection}[2]{\hspace{2em}\oldtocsubsection{#1}{#2}}
\renewcommand{\tocsubsubsection}[2]{\hspace{2em}\oldtocsubsubsection{#1}{#2}}

\numberwithin{equation}{section}

\newcommand{\Zcal}{\mathcal{Z}}
\newcommand{\Norm}{\operatorname{N}}
\newcommand{\Ecal}{\mathcal{E}}
\newcommand{\cores}{\operatorname{cores}}

\renewcommand{\H}{\operatorname{H}}

\begin{document}

\begin{abstract}
  In this paper, we develop a new necessary and sufficient condition for the vanishing of $4$-Massey products of elements in the mod-$2$ Galois cohomology of a field.
  This new description allows us to define a splitting variety for $4$-Massey products, which is shown in the Appendix to satisfy a local-to-global principle over number fields.
  As a consequence, we prove that, for a number field, all such $4$-Massey products vanish whenever they are defined.
  This provides new explicit restrictions on the structure of absolute Galois groups of number fields.
\end{abstract}
\maketitle

\tableofcontents

\section{Introduction}

Let $F$ be a field and $\bar F$ a separable closure of $F$.
The absolute Galois group of $F$, denoted $G_F := \Gal(\bar F/F)$ is an object of great interest in algebra and number theory.
Many aspects of modern Galois theory, in one way or another, aim to understand the structural properties of $G_F$.
Recent major results in Galois cohomology show that such absolute Galois groups are extremely rare among all profinite groups.
The most notable restriction on absolute Galois groups arises from the Bloch-Kato conjecture, which is now a theorem due to Rost-Voevodsky; see \cite{MR2811603} \cite{Rost-chain-lemma} \cite{MR2220090} \cite{MR2597737} \cite{zbMATH05594008}.
In particular, if $F$ contains a primitive $p$-th root of unity, then $\H^*(G_F,\z/p)$ is a \emph{quadratic algebra}.
More explicitly, this means that $\H^*(G_F,\z/p)$ is generated by elements of $\H^1(G_F,\z/p)$, and the relations are generated only by those relations appearing in degree $2$.
This is a very strong restriction on the group-theoretical structure of $G_F$.
Recently, other explicit structural restrictions on absolute Galois groups started to arise, based on the notion of \emph{Massey Products} in the context of Galois cohomology.

We will recall the definition of Massey products below, but we briefly note that, given $x_1,\ldots,x_n \in \H^1(G_F,\z/p)$ where $p$ is a prime, the $n$-Massey product, denoted $\langle x_1,\ldots,x_n \rangle$, is a (possibly empty) subset of $\H^2(G_F,\z/p)$.
In the case $n = 2$, one has a simple description in terms of the cup-product as $\langle x_1,x_2\rangle = \{x_1 \cup x_2 \}$.
Just as the cup-product $x_1 \cup x_2$ provides an obstruction to the existence of Heisenberg extensions of $F$ (of degree $p^3$), the $n$-Massey product provides as obstruction for the existence of higher $\z/p$-unipotent extensions.
In this respect, we are primarily interested in situations where the $n$-Massey product $\langle x_1,\ldots,x_n \rangle$ contains $0$.
In the sequel, when $\langle x_1,\ldots,x_n \rangle$ contains $0$, we will simply say that ``$\langle x_1,\ldots,x_n \rangle$ vanishes.''

In the breakthrough paper on the subject, Hopkins-Wickelgren \cite{hopkins} proved that, given $x_1,x_2,x_3 \in \H^1(G_F,\z/p)$, the triple Massey product $\langle x_1,x_2,x_3 \rangle$ always vanishes whenever it is nonempty, in the case where $F$ is a number field and $p = 2$.
This result was later extended by Min\'a\v{c}-T{\^a}n to arbitrary fields and $p = 2$ \cite{MT} partially based on ideas appearing in \cite{MR2018550}, and to arbitrary primes $p$ with $F$ a global field \cite{MR3452187}.

The end of 2014 and early 2015 saw a surge of activity on triple Massey products of elements of $\H^1(G_F,\z/p)$, significantly extending the results mentioned above.
Matzri \cite{Matzri} was first to announce his proof, extending these results to all primes $p$ and all fields $F$ which contain $\mu_p$.
Shortly thereafter, the arguments from {\em loc. cit.} were refined by Efrat-Matzri, and this work was eventually published in \cite{EM-triple}.
At around the same time as when \cite{EM-triple} was posted, Min\'a\v{c}-T{\^a}n \cite{MT2} also released their proof that triple Massey products of elements of $\H^1(G_F,\z/p)$ always vanish when defined, while also removing the condition that $F$ must contain $\mu_p$. 
Motivated by these results, Min\'a\v{c}-T{\^a}n \cite{MT} \cite{MT2} \cite{MT-kernel-unip} eventually formulated the so-called \emph{$n$-Massey Vanishing Conjecture}, which states that for an arbitrary field $F$, the $n$-Massey products of elements of $\H^1(G_F,\z/p)$ always vanish whenever they are non-empty.

The case of \emph{number fields} has always played a particularly important role in this context, as will be outlined in the historical discussion below.
In this respect, the present paper presents the first significant result concerning vanishing of $4$-Massey products in Galois cohomology.
Namely, this paper proves the following result.
\vskip 5pt
\noindent{\sc Main Theorem --} {\em Let $F$ be a number field, and let $x_1,x_2,x_3,x_4 \in \H^1(G_F,\z/2)$ be given.
If the Massey product $\langle x_1,x_2,x_3,x_4 \rangle$ is non-empty, then it vanishes.}
\vskip 5pt

To achieve this, we construct a ``splitting variety'' for the problem at hand, which works over \emph{any field} of characteristic $\neq 2$, and which is compatible with base-change.
More precisely, given a field $F$ of characteristic $\neq 2$ and $x_1,x_2,x_3,x_4 \in \H^1(G_F,\z/2)$, we construct an $F$-variety $\mathscr{X}_F$ such that the following are equivalent:
\begin{enumerate}
  \item The set of~$F$-rational points $\mathscr{X}_{F}(F)$ is non-empty.
  \item The $4$-Massey product $\langle x_1,x_2,x_3,x_4 \rangle$ vanishes.
\end{enumerate}
Furthermore, this construction is compatible with base-change, in the sense that for all $L/F$, one has $\mathscr{X}_F \otimes_F L = \mathscr{X}_L$. 

The study of the geometry and arithmetic of this variety falls within the reach of the recent results in~\cite{hw}.
Hence, when $F$ is a number field, it turns out to be possible to prove that (a variant of) $\mathscr{X}_F$ satisfies a local-to-global principle for the existence of rational points, as soon as the corresponding Brauer-Manin obstruction vanishes. 
%In other words, when $F$ is a number field, this variety has an $F$-rational point as soon as it has a rational point over any completion of $F$ which together make the Brauer-Manin obstruction vanish.
In the proof of Theorem \ref{thm-main-number-fields}, we arrange the existence of local points satisfying certain additional conditions which force the corresponding Brauer-Manin obstruction to vanish; this proves the mod-$2$ case of the $4$-Massey vanishing conjecture over number fields.
Moreover, it turns out that the Brauer-Manin obstruction \emph{always vanishes} in ``generic'' situations, which allows us to prove a stronger version of our result in such cases (see Theorem \ref{thm-main-intro}).

The overall strategy we take in this paper is similar to the approach taken by Wickelgren-Hopkins \cite{hopkins}, primarily because of the fact that we use splitting varieties.
However, their methods are highly specialized to the case of triple Massey products.
Indeed, working with $n$-Massey products for $n \geq 4$ is substantially harder, especially with respect to their definability and indeterminacy (see the discussion below). 
Moreover, there is a technical, yet fundamental difference between the two approaches arising from the additional conditions one must arrange for the local points of the splitting variety.
The details and new tools we develop here are therefore completely different and more technically involved than those in {\em loc.\ cit.}

\[ \star \star \star  \]

We want to say more about the context, and the history of the subject. Massey products were first introduced in the context of algebraic topology by W. Massey \cite{MR0098366}, as a collection of ``higher-order cohomology operations'' defined in terms of the cochain algebra.
For example, Massey products of elements of $\H^1$ play a central role in (rational) homotopy theory, as being the most basic obstruction to 1-formality of manifolds -- see the work of Deligne, Griffiths, Morgan, and Sullivan \cite{MR0382702} \cite{MR0646078} \cite{MR516917} \cite{MR876163}.
Massey products in Galois cohomology were first systematically considered over \emph{number fields} by Morishita \cite{MR1925911} \cite{MR2004124}, Vogel \cite{MR2132673} and Sharifi \cite{MR2312552}.
These papers were primarily focused on Galois groups of extensions with restricted ramification over number fields, where Massey products have interesting connections with topics from Iwasawa theory, Milnor invariants, and R\'edei symbols.
Understanding Massey products in the Galois cohomology of number fields is particularly important, as they also show up in the context of Grothendieck's \emph{section conjecture} -- see Wickelgren \cite{wickelgren2009lower} \cite{MR3220523} \cite{MR3051256}.
More generally, Massey products in Galois cohomology play an important role in understanding the structure of nilpotent quotients of absolute Galois groups.
Therefore, a detailed understanding of Massey products in Galois cohomology could lead to significant generalizations of results in \cite{MR2863369} \cite{MR1405942} from the two-step nilpotent setting to more general settings.

The investigation of Massey products in Galois cohomology of arbitrary fields has recently started progressing very rapidly.
This surge started with the work of Hopkins-Wickelgren \cite{hopkins}, and further progressed by Min\'a\v{c}-T{\^a}n \cite{MTlocal} \cite{MT} \cite{MT2} \cite{MT3} and Efrat-Matzri \cite{MR3415664} \cite{EM-triple} \cite{MR3239143} \cite{Matzri}.
In fact, ideas related to vanishing of mod-$2$ triple Massey products already appeared in 2003 by Gao-Leep-Min\'a\v{c}-Smith \cite{MR2018550}, albeit using different terminology.
However, it is important to note that all of the results mentioned above were restricted to studying \emph{triple Massey products}.

Until now, the cases of the $n$-Massey vanishing conjecture for $n \geq 4$ have remained completely open.
Already in the case $n = 4$, having a \emph{non-empty} $4$-Massey product forces the vanishing of a new invariant defined by Isaksen \cite{MR3314129}.
In particular, prior to the present paper there was not even a strategy for approaching the $4$-Massey vanishing conjecture.

\[ \star \star \star  \]

We now introduce the basic concepts, and the precise notation, needed to state the main results of the paper.

\subsection{Basic Notation}
Throughout the paper, $F$ will denote a field of characteristic $\neq 2$.
We will use the usual notation $\H^*(F,A) := \H^*(G_F,A)$ for the Galois cohomology of $F$ with coefficients in the $G_F$-module $A$.

Recall that Kummer theory yields a canonical isomorphism
\[ F^\times/F^{\times 2} \xrightarrow{\cong} \H^1(F,\F_2). \]
For an element $x \in F^\times$, we write $[x]$ for its class in $F^\times/F^{\times 2}$, and $\chi_x \in \H^1(F,\F_2)$ for the image of $[x]$ under the Kummer isomorphism.
We will usually consider $\chi_x$ as a (continuous) homomorphism
\[ \chi_x : G_F \rightarrow \F_2 \]
via the canonical identification $\H^1(F,\F_2) = \Hom^{\rm cont}(G_F,\F_2)$.

Given $a,b \in F^\times$, we will usually write $(a,b)_F$ (or $(a,b)$ when $F$ is understood) for the cup-product $\chi_a \cup \chi_b \in \H^2(F,\F_2)$.
This notation borrows from the fact that $\H^2(F,\F_2)$ is canonically isomorphic to the $2$-torsion of $\operatorname{Br}(F)$, and that the class of the quaternion algebra $(a,b)_F$ corresponds to $\chi_a \cup \chi_b$ via this identification.

\subsection{The Groups $\U_n(\F_2)$}

The group $\U_n(\F_2)$, for $n \geq 2$, is comprised of the $n\times n$ upper-triangular matrices with entries in $\F_2$ with $1$'s along the diagonal.
The group $\U_{n}(\F_2)$ is endowed with $n-1$ homomorphisms
\[ s_1,\ldots,s_{n-1} : \U_{n}(\F_2) \rightarrow \F_2 \]
defined as $s_i(g) = g_{i,i+1}$ (the $i$-th near-diagonal component of $g$).

The center $\Zcal(\U_{n}(\F_2))$ of $\U_{n}(\F_2)$ consists of those matrices whose only possibly non-zero coefficient above the diagonal is in the top-right corner.
In particular, the map $g \mapsto g_{1,n}$ induces an isomorphism $\Zcal(\U_{n}(\F_2)) \cong \F_2$.
We write $\overline \U_{n}(\F_2) := \U_{n}(\F_2)/\Zcal(\U_{n}(\f_2))$, and consider $\U_{n}(\F_2)$ as an extension of $\overline\U_{n}(\F_2)$ by $\F_2$.
Furthermore, we denote by $\xi_n$ the element of $\H^2(\overline\U_{n}(\F_2),\F_2)$ associated to this extension.

\subsection{Massey Products}

Let $\Gamma$ be a profinite group, and let $x_1,\ldots,x_n \in \H^1(\Gamma,\F_2)$ be given.
In this context, we say that the $n$-Massey product $\langle x_1,\ldots,x_n \rangle$ is {\bf defined} provided that there exists a homomorphism $\phi : \Gamma \rightarrow \overline\U_{n+1}(\F_2)$ such that $x_i = s_i \circ \phi$ for $i = 1,\ldots,n$.
Furthermore, in this case we say that $\phi$ is a {\bf defining system} for the $n$-Massey product $\langle x_1,\ldots,x_n \rangle$.

The {\bf $n$-Massey product} associated to the defining system $\phi$, denoted by $\langle x_1,\ldots,x_n \rangle_\phi$, is defined to be $\phi^*\xi_{n+1} \in \H^2(\Gamma,\F_2)$, the pull-back of $\xi_{n+1}$ along $\phi$.
Note that one has $\langle x_1,\ldots,x_n \rangle_\phi = 0$ if and only if the map $\phi : \Gamma \rightarrow \overline\U_{n+1}(\F_2)$ lifts to a homomorphism $\tilde\phi : \Gamma \rightarrow \U_{n+1}(\F_2)$.

Finally, the {\bf $n$-Massey product} $\langle x_1,\ldots,x_n \rangle$ is defined as the set
\[ \langle x_1,\ldots,x_n \rangle := \{\langle x_1,\ldots,x_n \rangle_\phi\} \]
where $\phi$ varies over all defining systems for $\langle x_1,\ldots,x_n \rangle$.
In particular, the $n$-Massey product $\langle x_1,\ldots,x_n \rangle$ is non-empty if and only if it is defined.
As mentioned above we will be primarily interested in situations where the $n$-Massey product $\langle x_1,\ldots,x_n \rangle$ contains $0$, and we say that ``$\langle x_1,\ldots,x_n \rangle$ vanishes'' in such situations.
Note that, when we say ``$\langle x_1,\ldots,x_n \rangle$ vanishes'' we are also implying that $\langle x_1,\ldots,x_n \rangle$ is defined (as $\langle x_1,\ldots,x_n \rangle$ is non-empty).

\begin{rmk2}
  We have presented the definition of defining systems and Massey products in the context of group-cohomology from the point of view of \emph{embedding problems}.
  This is nevertheless equivalent to the classical (highly technical) definitions, by the work of Dwyer \cite{MR0385851}.
  For our purposes, Massey products are defined as above.
\end{rmk2}

We will simplify the notation somewhat in the context of Galois cohomology.
Namely, given $a_1,\ldots,a_n \in F^\times$, we write $\langle a_1,\ldots,a_n \rangle$ instead of $\langle \chi_{a_1},\ldots,\chi_{a_n} \rangle$.
We will follow this convention when talking about defining systems as well as Massey products themselves.

\subsection{Main Results}

We are now prepared to state our main theorems which characterize the vanishing of $4$-Massey products in mod-$2$ Galois cohomology.

\begin{maintheorem} \label{thm-main-four-cup-products-without-tildes}
  Let $F$ be a field of characteristic $\neq 2$.
  Let $a,b,c,d \in F^\times$ be given, choose square roots $\sqrt{a}$ resp.\ $\sqrt{d}$ of $a$ resp.\ $d$ in an algebraic closure of $F$, and put $E := F[\sqrt{a},\sqrt{d}]$.
  Then the following are equivalent:
  \begin{enumerate}
    \item The $4$-Massey product $\langle a,b,c,d \rangle$ vanishes (i.e., it is defined and contains $0$).
    \item There exist $B \in F[\sqrt{a}]$, $C \in F[\sqrt{d}]$ and $z_1,z_2 \in F^\times$ such that the following conditions hold:
    \begin{enumerate}
      \item One has $\Norm_{F[\sqrt{a}]/F}(B) = b \cdot z_1^2$ and $\Norm_{F[\sqrt{d}]/F}(C) = c \cdot z_2^2$.
      \item One has $(B,C)_{E} = 0$, $(B,c)_{F[\sqrt{a}]} = 0$, $(b,C)_{F[\sqrt{d}]} = 0$ and $(b,c)_F = 0$.
    \end{enumerate}
    \item There exist $B \in F[\sqrt{a}]$, $C \in F[\sqrt{d}]$ and $z_1,z_2 \in F^\times$ such that the following conditions hold:
    \begin{enumerate}
      \item One has $\Norm_{F[\sqrt{a}]/F}(B) = b \cdot z_1^2$ and $\Norm_{F[\sqrt{d}]/F}(C) = c \cdot z_2^2$.
      \item One has $(B,C)_{E} = (B,c)_{E} = (b,C)_{E} = (b,c)_E = 0$.
    \end{enumerate}
  \end{enumerate}
\end{maintheorem}

Note that condition (2) of Theorem \ref{thm-main-four-cup-products-without-tildes} can be readily described in terms of polynomial equations over $F$, hence defining an (affine) $F$-variety.
Theorem \ref{thm-main-four-cup-products-without-tildes} then shows that this variety has an $F$-point if and only if $\langle a,b,c,d \rangle$ vanishes.
Note, however, that these equations \emph{depend} on whether $a$, $d$, and/or $ad$ are squares in $F$ (the definition involves a Weil-restriction from $F[\sqrt{a},\sqrt{d}]$ to $F$); in other words, the variety is not compatible with base-change to extensions of $F$.
This is undesirable, as compatibility with base-change will be an important property towards the end of the paper.
With some additional work, we are able to obtain the following characterization theorem which provides us with our desired uniform polynomial equations.

\begin{maintheorem}
  \label{thm-main-fundamental-equations}
  Let $F$ be a field of characteristic $\neq 2$ and let $a,b,c,d \in F^\times$ be given.
  Consider the finite \'etale $F$-algebra
  \[ \Ecal := F[X,Y]/(X^2-a,Y^2-d). \]
  Then the following are equivalent:
  \begin{enumerate}
    \item The $4$-Massey product $\langle a,b,c,d \rangle$ vanishes.
    \item There exist $x_1,y_1,x_2,y_2 \in F$, $z_1,z_2 \in F^\times$ and $u,v \in \Ecal$ such that the following equations are satisfied:
    \begin{enumerate}
      \item One has $x_1^2 - y_1^2 \cdot a = b \cdot z_1^2$ and $x_2^2 - y_2^2 \cdot d = c \cdot z_2^2$.
      \item One has $u^2-\tilde B v^2 = \tilde C$ in $\Ecal$, where $\tilde B = x_1 + y_1 \cdot X \in \Ecal$ and $\tilde C = x_2 + y_2 \cdot Y \in \Ecal$.
    \end{enumerate}
  \end{enumerate}
\end{maintheorem}

Note that the polynomial equations described by condition (2) of Theorem \ref{thm-main-fundamental-equations} actually have the same shape over any field which contains $a,b,c,d$.
The $F$-variety defined by these equations is what we will eventually call {\bf the splitting variety for $\langle a,b,c,d \rangle$}.

It is important to note that both Theorems \ref{thm-main-four-cup-products-without-tildes} and \ref{thm-main-fundamental-equations} will play a key role in this paper.  Indeed, condition (2) of Theorem \ref{thm-main-four-cup-products-without-tildes} has an immediate and direct formulation involving cup-products in mod-2 Galois cohomology (we exhibit some direct applications, over any field, in \S\ref{sec-first-applications}).  On the other hand, condition (2) of Theorem \ref{thm-main-fundamental-equations} defines a splitting variety whose geometry is remarkably simple. In generic situations, it satisfies the Hasse principle for the existence of rational points; in all cases, the local-to-global principle is governed by the Brauer-Manin obstruction, which takes a simple form here. See Theorem~\ref{app:th:numberfield} in the Appendix for a detailed statement. We thereby obtain a more precise version of the Theorem announced above:

\begin{maintheorem} 
  \label{thm-main-intro}
  Let $F$ be a number field, and let $a,b,c,d \in F^\times$ be given.
  Then the following are equivalent:
  \begin{enumerate}
    \item The $4$-Massey product $\langle a,b,c,d \rangle$ vanishes.
    \item The $4$-Massey product $\langle a,b,c,d \rangle$ is defined.
  \end{enumerate}
  If furthermore $ad$, $ab$, $cd$ are all non-squares in $F$, then the above conditions are further equivalent to:
  \begin{enumerate}
    \item[(3)] One has $(a,b)_F = (b,c)_F = (c,d)_F = 0$.
  \end{enumerate}
\end{maintheorem}

%It is natural to ask whether the implication $(3) \implies (1)$ holds for all fields. More generally, one is tempted to formulate the {\em Strong Massey Vanishing Conjecture}, for all~$n \ge 3$, stating that whenever~$(a_i, a_{i+1})_F  = 0$, for~$1 \le i < n$, then the Massey product~$\langle a_1, \ldots, a_n \rangle$ is defined and vanishes. (For~$n=3$ there is no difference between the usual conjecture and the strong one.)
Theorem \ref{thm-main-intro} will be proved in Theorems \ref{thm-main-number-fields}, \ref{thm-main-number-fields-generic} below. 
It is natural to ask whether the implication $(3) \implies (1)$ holds in general.
It turns out that this implication \emph{fails in general}, even over number fields.
See Remark \ref{rmk-generic}, Example \ref{app:example}, and the surrounding discussions for more details.

\[ \star \star \star  \]
   {\em Organization of the paper.} After some preliminaries in the next section, we prove Theorem~\ref{thm-main-four-cup-products-without-tildes} in~\S\ref{sec-cup-products}.
   In section \ref{sec-first-applications}, we give some first applications of Theorem \ref{thm-main-four-cup-products-without-tildes} by proving a few cases of the 4-Massey vanishing conjecture by hand, over arbitrary fields.
   Then in~\S\ref{sec-splitting} we introduce the splitting variety $\mathscr{X}_F$, as well as a variant $X_F$ which will simplify some calculations. 
   The next section, that is~\S\ref{sec-applications}, gives a proof of Theorem~\ref{thm-main-intro}.
   Finally in~\S\ref{sec-explicit} we make some of our constructions explicit, and explain concretely how to get a Galois extension with group~$\uf$ when~$\langle a, b, c, d \rangle$ vanishes and~$a,b, c, d$ are linearly independent modulo squares; incidentally, this gives an alternative, more pedestrian proof for the implication $(3) \implies (1)$ in Theorem~\ref{thm-main-four-cup-products-without-tildes} in this case.

   An Appendix by Wittenberg shows that the variety $X_F$ satisfies the local-to-global principle alluded to above, which is of course a crucial ingredient for Theorem~\ref{thm-main-intro}.

   \bigskip

   \noindent {\em Ackowledgements.} First and foremost, we have been in touch very frequently with Olivier Wittenberg while writing this paper. We thank him warmly for writing the Appendix, and for various suggestions. We also thank Adriano Marmora for suggesting us to contact Olivier in the first place.
   Moreover, we warmly thank Nguy\~{\^{e}}n Duy T{\^a}n for his interest and insight, and for his extremely inspiring collaboration with the second author.
   It is a great pleasure to also thank B.\  I.\  Chetard, I.\  Efrat, Ch.\  Kapulkin, E.\  Matzri, C.\  McLeman, D.\  Neftin, M.\  Palaisti, C.\  Quadrelli and K.\  Wickelgren for inspiring discussions.
   Finally, we thank the referee for helpful comments regarding the exposition.

\section{Preliminaries} \label{sec-defs}

\pierre{No need for the \S~ on notation, it's all in the intro. Also the def of~$F_a$ will be at the beginning of the next section.}

\subsection{The groups~$\U_3(\f_2)$ and $\uf$} \label{subsec-notation-Un}
 We shall need special notation for these two groups. First note that we {\em define} the dihedral group of order~$8$ to be~$\U_3(\f_2)$, and we may write~$D_4 = \U_3(\f_2)$. An element~$g \in D_4$ is a matrix of the form
\[ g = \left(\begin{array}{ccc}
  1 & s_1(g) & t(g) \\
  0 & 1 & s_2(g) \\
  0 & 0 & 1
\end{array}\right) \, .   \]
Thus~$D_4$ is equipped with maps~$s_1, s_2, t \colon D_4 \to \f_2$. (The letter~$t$ is for ``top''.) Note that the first two are group homomorphisms, but~$t$ is not. Our favourite generators are the involutions~$\sigma_1$ and~$\sigma_2$, with~$s_i(\sigma_i) = 1$ and~$s_j(\sigma_i) = t (\sigma_i) = 0$ for~$j \ne i$.

Similarly, an element~$g \in \uf$ will be written
\[ \left(\begin{array}{ccccc}
  1 & s_1(g) & t_1(g) & u_1(g) & z(g) \\
  0 & 1 & s_2(g) & u_3(g) & u_2(g) \\
  0 & 0 & 1 & s_3(g) & t_2(g) \\
  0 & 0 & 0 & 1 & s_4(g) \\
  0 & 0 & 0 & 0 & 1
\end{array}\right) \, .  \]
This endows~$\uf$ with maps~$s_1, \ldots, z \colon \uf \to \f_2$, and~$s_i$ is a group homomorphism for~$1 \le i \le 4$.

More generally the group~$\U_n(\f_2)$ has homomorphisms~$s_i \colon \U_n(\f_2) \to \f_2$ for~$1 \le i \le n-1$, already mentioned in the Introduction, obtained by looking at the entries on what we call the near-diagonal. If we define elements~$\sigma_i$ by requiring~$s_i(\sigma_i)= 1$ while all the other entries of~$\sigma_i$ above the diagonal are~$0$, then each~$\sigma_i$ is an involution, and these generate~$\U_n(\f_2)$.

We note that~$\U_n(\f_2)$ has an automorphism which exchanges~$\sigma_i$ with~$\sigma_{n-i}$. Most of our considerations respect this symmetry, and this motivates the notation above for~$\uf$. (The automorphism is given by ``the transpose but along the other diagonal'', followed by~$g \mapsto g^{-1}$.)

\subsection{Around the group~$D_4$}

We write~$s= (s_1, s_2) \colon D_4 \to C_2 \times C_2$, where we have identified~$\f_2$ with the cyclic group of order~$2$ in multiplicative notation, written~$C_2$. There is an exact sequence
\[ 1 \longrightarrow \f_2 \longrightarrow D_4 \stackrel{s}{\longrightarrow} C_2 \times C_2 \longrightarrow 1 \, ,  \]
the kernel of~$s$ being generated by~$[\sigma_1, \sigma_2]$ (which is the element~$g$ with~$t(g)= 1$ and~$s_i(g) = 0$, $i= 1, 2$).

The quotient group~$C_2 \times C_2$ is generated by the images of~$\sigma _1, \sigma _2$, written~$\bar \sigma_1, \bar \sigma_2$. The cohomology group~$\h^1(C_2^2, \f_2) = \Hom(C_2^2, \f_2)$ is endowed with the dual basis~$\bar s_1, \bar s_2$. The next Lemma is very well-known:

\begin{lem} \label{lem-coho-class-D4}
The cohomology class of the above extension is~$\bar s_1 \bar s_2 \in \h^2(C_2^2, \f_2)$. \hfill $\square$
\end{lem}

%% \begin{proof}
%%   This is very well-known, but since the argument is a warm-up for a more complicated one to come, we go through it quickly. The cohomology class is a homogeneous polynomial~$P$ in the two variables~$\bar s_1, \bar s_2$, and we can derive information on~$P$ by looking at the inverse images of various subgroups.

%%   For example $s^{-1}(\langle \sigma_1 \rangle) \cong C_2^2$, in fact it is the group~$E_1$ studied below. This shows that the cohomology class restricts to~$0$ in $\h^2(\langle \sigma_1 \rangle, \f_2)$, or that~$P(\bar s_1, 0) = 0$. In other words, the monomial~$\bar s_1^2$ does not appear in~$P$. We eliminate likewise~$\bar s_2^2$.

%%   The class cannot be zero since the extension is not split, so the lemma must hold.
%% \end{proof}

We introduce the two elementary abelian subgroups~$E_1, E_2$, where~$E_1$ is the kernel of~$s_2$ and~$E_2$ is the kernel of~$s_1$ (the switch is justified by the next Lemma). A very useful observation is that~$t$, when restricted to either of these, is a group homomorphism.

\begin{lem} \label{lem-transfer}
The corestriction
\[ \cores \colon \h^1(E_i, \f_2) \longrightarrow \h^1(D_4, \f_2)  \]
carries~$t|_{E_i}$ to~$s_i$, for~$i= 1, 2$. More generally if~$\Gamma $ is any profinite group with a continuous homomorphism~$\phi \colon \Gamma \to D_4$, and if~$H = \phi^{-1}(E_i)$ is assumed to have index~$2$ in~$\Gamma $, then the corestriction
\[\cores \colon \h^1(H, \f_2) \longrightarrow \h^1(\Gamma , \f_2)  \]
carries~$t \circ \phi$ to~$s_i \circ \phi$, for~$i= 1, 2$.
\end{lem}

\begin{proof}
We recall some properties of the corestriction
\[ \cores \colon \h^i(N, M) \longrightarrow  \h^i(G, M)  \]
where~$N$ is a subgroup of finite index of the profinite group~$G$, and~$M$ is a~$G$-module. In fact, we only need to consider the case when~$M$ has a trivial action, $N$ is closed of index~$2$ (and thus is normal) in~$G$, and~$i=1$, so that
\[ \cores \colon \Hom(N, M) \longrightarrow \Hom(G, M) \, .    \]
Here if~$f \colon N \to M$, then the map~$\cores(f) \colon G \to M$ is characterized as follows. Pick~$\tau \in G \smallsetminus N$. Then (i) $\cores(f)(n) = f(n) + f(\tau^{-1} n \tau )$ for $n \in N$, and (ii) $\cores(f)(\tau ) = f( \tau ^2)$. This follows from the material in~\cite{tate}, \S2, for example.

Let us use this for~$N= E_1$ and~$G= D_4$. If
\[ n= \left(\begin{array}{rrr}
1 & a & b \\
0 & 1 & 0 \\
0 & 0 & 1
\end{array}\right)  \quad\textnormal{and}\quad \tau = \left(\begin{array}{rrr}
1 & c & d \\
0 & 1 & 1 \\
0 & 0 & 1
\end{array}\right) \, ,  \]
then
\[ \tau ^{-1} n \tau = \left(\begin{array}{rrr}
1 & a & a + b \\
0 & 1 & 0 \\
0 & 0 & 1
\end{array}\right)  \, .  \]
Moreover
\[ \tau^2 =  \left(\begin{array}{rrr}
1 & 0 & c \\
0 & 1 & 0 \\
0 & 0 & 1
\end{array}\right) \, .  \]
Thus~$t(n) + t(\tau ^{-1} n \tau ) = b+a+b=a = s_1(n)$, and $t(\tau^2) = c = s_1(\tau )$. Hence the first statement of the lemma for~$i=1$. The other cases are treated similarly.
\end{proof}

\subsection{$D_4$-extensions of fields}
We proceed to apply the above observations in a Galois-theoretic context, but a couple of comments are in order.
First, the group~$D_4$ has an automorphism exchanging~$\sigma_1$ and~$\sigma_2$, but the Proposition below is not ``symmetric'' in this way -- it involves the subgroup~$E_2$ and not~$E_1$, for example (so that one could get a new Proposition by exchanging the roles of various players).
Second, when asked for a basis for~$E_2$, the reader would probably offer~$[\sigma_1, \sigma_2], \sigma_2$; however, later considerations with~$\uf$ compel us to work with~$\sigma_2[\sigma_1, \sigma_2], \sigma_2$ instead (specifically, we want Lemma~\ref{lem-classes-extensions-of-N} to have the simple form given below).
This is reflected in the Proposition below, since the dual basis of~$\h^1(E_2, \f_2)$ is~$t, s_2+t$ (or, in more complete notation, $t|_{E_2}, s_2|_{E_2} + t|_{E_2}$).

\begin{prop} \label{prop-D4-iff-cup-prop-iff-norm}
Let $F$ be a field of characteristic $\neq 2$ and let $a,b \in F^\times$ be given.
Then the following are equivalent:
\begin{enumerate}
\item $(a, b)_F= 0$.
\item There exists a continuous homomorphism~$\phi \colon G_F \longrightarrow D_4$ such that~$s_1 \circ \phi = \chi_a$ and~$s_2 \circ \phi = \chi_b$.
\item There exist~$x, y, z \in F$ such that~$x^2 - a y^2 = b z^2$, with~$z \ne 0$.
\end{enumerate}
When the equivalent conditions hold and $B := x + y \sqrt{a}$, we will say that $\phi$ from (2) and $x,y,z$ from (3) are \emph{consistent}, provided that $\chi_{bB} = t \circ \phi$ and $\chi_B = (s_2 + t) \circ \phi$ as elements of $\H^1(F[\sqrt{a}], \f_2)$.
Then given any $\phi$ as in (2), we can choose $x,y,z$ as in (3) which are consistent with $\phi$.
Conversely, given any $x,y,z$ as in (3), we can choose $\phi$ as in (2) which is consistent with $x,y,z$.

\end{prop}

Again, this is essentially known, but we need the precise version given here. Note that it is necessary to deal with the cases when either~$a$ or~$b$ is a square, and that the proof below gives more concrete information in some situations. Also note that, as promised, the elements~$t, s_2+t$, related to~$E_2$, make an uncanny appearance.

\begin{proof}
  We can combine~$\chi_a$ and~$\chi_b$ into a homomorphism~$G_F \to C_2 \times C_2$. The obstruction to lifting it to~$D_4$ is the cohomology class of the extension, so that Lemma~\ref{lem-coho-class-D4} gives immediately the equivalence of (1) and (2).

We first conduct the rest of the proof under the following

\medskip

{\em Assumption.} Assume for the moment that~$a$ is not a square in~$F$.

\medskip

Suppose (2) holds. Note that~$\phi^{-1}(E_2) = G_{F[\sqrt a]}$. Let~$B' = x' + y' \sqrt a \in F[\sqrt a]$ be such that $\chi_{B'} = t \circ \phi |_{G_{F[\sqrt a]}}$. Since~$a$ is not a square in~$F$, the subgroup~$G_{F[\sqrt a]}$ has index~$2$ in~$G_F$. Lemma~\ref{lem-transfer} shows then that~$\cores(\chi_{B'}) = s_2 \circ \phi = \chi_b$. Now recall that under the identifications of~$\h^1(F, \f_2)$ with~$\li F$ and of~$\h^1(F[\sqrt a], \f_2)$ with~$\li{F[\sqrt a]}$, the corestriction becomes the usual norm~$N_{F[\sqrt a]/F}$. It follows that
\[ N_{F[\sqrt a]/F}(B') = (x')^2 - a (y')^2 = b ~\textnormal{mod squares} \, .   \]
This gives (3), clearly, but we need to modify~$B'$ to get the consistency statement. And indeed, we put~$B = bB'$, so that~$B' = bB$ modulo squares, and the result follows.

Next, we must prove that (3) implies (2), or equivalently (1). We may as well suppose that~$a$ and~$b$ are both not squares, for (1) holds trivially otherwise. The assumption is that~$b$ is a norm from~$F[\sqrt a]$, or in more cohomological terms, that~$\chi_b$ is the corestriction of an element from the subgroup~$G_{F[\sqrt a]}$. That the cup product~$\chi_a \chi_b = 0$ then follows from the Arason exact sequence \cite{arason}.

However, to prove the claimed consistency, a more explicit argument is needed. Assume~$b$ is not a square. The element~$B= x + y \sqrt a$, where~$x, y$ are as in (3), is fixed up to squares by~$\gal(F[\sqrt a, \sqrt b]/F)$, as is readily checked. It follows that~$K=F[\sqrt a, \sqrt b, \sqrt{B}]$ is Galois over~$F$ (by equivariant Kummer theory, if you will).

We distinguish two cases, and assume first that~$a$ and~$b$ are not equal modulo squares.
We now introduce elements~$\sigma_1, \sigma_2 \in \gal(K/F)$ which are dual to~$\sqrt a, \sqrt b$ in the obvious sense.
Direct computation shows that~$\sigma_1^2 = \sigma_2^2 = 1$, and that~$[\sigma_1, \sigma_2] (\sqrt B) = -\sqrt B$, so that~$[\sigma_1, \sigma_2] \ne 1$.
From this one draws readily that~$\gal(K/F) \cong D_4$ and (again!) that (2) holds. Also, one computes that~$\sigma_2(\sqrt B) = \pm \sqrt B$.
We want to ensure that~$\sigma_2(\sqrt B) = - \sqrt B$, and to achieve this we replace~$\sigma_2$ by~$\sigma_2 [\sigma_1, \sigma_2]$ if needed. Having done this, the elements~$\sqrt B$, $\sqrt{bB}$ are dual to~$\sigma_2$, $\sigma_2 [\sigma_1, \sigma_2]$, and one checks that the corresponding map~$\phi\colon G_F \to D_4$ has precisely the required consistency.

If~$a=b$ modulo squares, one sees that~$\gal(K/F)$ has order 4, so is abelian, and if~$\sigma $ is the non-trivial element of~$\gal(F[\sqrt a]/F)$ extended to~$\gal(K/F)$, another direct calculation shows that~$\sigma $ does not have order~$2$. So~$\gal(K/F) \cong C_4$ can be identified with the subgroup of~$D_4$ generated by~$\sigma_1 \sigma_2$, and (2) follows. Consistency is automatic.

Finally, if~$b$ is a square in~$F$, we use the same extension~$K= F[\sqrt a, \sqrt B]$ of~$F$, but compute that~$\gal(K/F) \cong C_2^2$. We identify this group with~$E_1$ appropriately, yielding a consistent~$\phi$.

\medskip

{\em The case when~$a$ is a square.}

\medskip

In this situation (1) holds trivially, and thus (2) also holds. As for (3), if~$a= u^2$ then put
\[ x_0= \frac{b+1} {2} \quad\textnormal{and}\quad  y_0 = \frac{1-b} {2u}\]
and compute that~$x_0^2 - a y_0^2 = b$.

Let us see how we can adjust~$\phi$ from~$x, y, z$. Put~$B = x + y \sqrt a \in F$, and consider the characters~$\chi_{bB}$ and~$\chi_b$, together defining a homomorphism~$G_F \to C_2 \times C_2$. Identifying Klein's group with~$E_2$ sitting in~$D_4$ appropriately, we obtain~$\phi \colon G_F \to D_4$ satisfying our requirements.

Our very last step is to see how one can adjust~$x, y, z$ from~$\phi$. First put~$B_0 = x_0 + y_0 \sqrt a \in F$ where~$x_0$ and~$y_0$ are as above, which is a non-zero element since~$B_0(x - y\sqrt a) = b \ne 0$.
For~$f \in F$, put~$x= \frac{f} {B_0} x_0$ and~$y = \frac{f} {B_0} y_0$, so that~$x^2 - a y^2 = bz^2$ for some~$z \in F^\times$, while~$B = x + y \sqrt a = f$, an arbitrary element of~$F$.
Of course~$\phi $ lands in the subgroup~$E_2$, so we only need to pick~$f$ so that~$\chi_f = (s_2 + t) \circ \phi = \chi_B$; we have then (2) and (3) simultaneously and consistently.
\end{proof}

\subsection{The group~$\uf$ and its subquotients} \label{subsec-U5-subgroups}

Let~$S$ (for ``square'') be the subgroup of matrices of the form
\[ \left(\begin{array}{rrrrr}
1 & 0 & 0 & y_1 & y_2 \\
0 & 1 & 0 & y_3 & y_4 \\
0 & 0 & 1 & 0 & 0 \\
0 & 0 & 0 & 1 & 0 \\
0 & 0 & 0 & 0 & 1
\end{array}\right) \, .
\]
Then~$S \cong C_2^4$, and our favourite $\F_2$-basis, denoted $e_1,e_2,e_3,e_4$, will be given by
\[ e_1 = e = \left(\begin{array}{rrrrr}
1 & 0 & 0 & 0 & 0 \\
0 & 1 & 0 & 1 & 0 \\
0 & 0 & 1 & 0 & 0 \\
0 & 0 & 0 & 1 & 0 \\
0 & 0 & 0 & 0 & 1
\end{array}\right)  , ~ e_2 = \sigma_1 e \sigma_1^{-1} = \left(\begin{array}{rrrrr}
1 & 0 & 0 & 1 & 0 \\
0 & 1 & 0 & 1 & 0 \\
0 & 0 & 1 & 0 & 0 \\
0 & 0 & 0 & 1 & 0 \\
0 & 0 & 0 & 0 & 1
\end{array}\right) \, ,   \]
\[ e_3 = \sigma_4 e \sigma_4^{-1} = \left(\begin{array}{rrrrr}
1 & 0 & 0 & 0 & 0 \\
0 & 1 & 0 & 1 & 1 \\
0 & 0 & 1 & 0 & 0 \\
0 & 0 & 0 & 1 & 0 \\
0 & 0 & 0 & 0 & 1
\end{array}\right)  , ~ e_4 = (\sigma_1 \sigma_4) e (\sigma_1 \sigma_4)^{-1} = \left(\begin{array}{rrrrr}
1 & 0 & 0 & 1 & 1 \\
0 & 1 & 0 & 1 & 1 \\
0 & 0 & 1 & 0 & 0 \\
0 & 0 & 0 & 1 & 0 \\
0 & 0 & 0 & 0 & 1
\end{array}\right) \, .  \]
The centralizer of~$S$ in~$\uf$, which we denote by~$\cent(S)$, is easily seen to be comprised of the elements~$g$ for which~$s_1(g) = s_4(g)= 0$, that is~$\cent(S) = \ker s_1 \cap \ker s_4$.
In particular~$\sigma_2$ and~$\sigma_3$ centralize~$S$, and from the formulae above we see that~$S$ is normal in~$\uf$. We shall write~$G= \uf / S$, which we identify with~$D_4 \times D_4$, as we visibly may. The image of~$\cent(S)$ in~$G$, that is~$\cent(S)/S$, will be denoted by~$N$, a normal subgroup of~$G$.

One has~$G/N = \uf/\cent(S) = \langle \sigma_1, \sigma_4 \rangle \cong C_2^2$ (we shall often write~$\sigma_i$ for the image of this element in various quotients, whenever no confusion can arise). The next observation is now clear, but it is crucial:

\begin{lem} \label{lem-S-is-free}
The action of~$G/N$ on~$S$, induced by conjugation, turns it into a free $\f_2[G/N]$-module of rank $1$. A specific isomorphism~$\f_2[G/N] \to S$ is given by $1 \mapsto e$, for example. \hfill $\square$
\end{lem}

The group~$N$ itself also has a simple structure: one has~$N \cong C_2^4$, a basis being~
\[ \sigma_2[\sigma_1, \sigma_2], \ \sigma_2, \ \sigma_3, \ \sigma_3 [\sigma_4, \sigma_3]. \]
(A choice which respects the ambient ``symmetry'' already alluded to.)
The corresponding dual basis of $\h^1(N, \f_2)$ will be denoted~$x_1, x_2, x_3, x_4$.
To bridge the notation with that of the previous sections, we regard~$N$ as sitting in~$D_4 \times D_4$, which itself possesses six maps~$s_1, s_2, t_1, s_3, s_4, t_2$ to~$\f_2$, using names adapted from~\S\ref{subsec-notation-Un}.
With this notation, one has~$x_1= t_1$, $x_2= s_2+t_1$, $x_3= s_3+t_2$, and $x_4 = t_2$ (where restrictions to~$N$ are implicit).

Next we introduce some subgroups of~$S$, and use them to produce extensions of~$N$. These extensions turn out to control the entire situation, as will be explained. So we let
\[ S_1 := \langle e_2, e_3, e_4\rangle \, , ~S_2 := \langle e_1, e_3, e_4\rangle, ~S_3 := \langle e_1, e_2, e_4\rangle \, .   \]
(The group~$S_4$ which could be defined using the same logic will not play any role, as it happens. Also note that among these three, only~$S_1$ respects the ambient ``symmetry''.)

\begin{lem} \label{lem-classes-extensions-of-N}
Using the notation above, the following hold:
\begin{enumerate}
  \item The cohomology class of the extension
    \[ 0 \longrightarrow S/S_1 \cong \f_2 \longrightarrow \cent(S)/S_1 \longrightarrow N \longrightarrow 1  \]
    is~$x_2x_3$.
  \item The cohomology class of the extension
    \[ 0 \longrightarrow S/S_2 \cong \f_2 \longrightarrow \cent(S)/S_2 \longrightarrow N \longrightarrow 1  \]
    is~$x_1x_3$.
  \item The cohomology class of the extension
    \[ 0 \longrightarrow S/S_3 \cong \f_2 \longrightarrow \cent(S)/S_3 \longrightarrow N \longrightarrow 1  \]
    is~$x_2x_4$.
\end{enumerate}
\end{lem}

\begin{proof}
An element of~$\cent(S)$ has the form 
\[ g = \left(\begin{array}{rrrrr}
1 & 0 & t_{1}(g) & u_{1}(g) & z(g) \\
0 & 1 & s_{2}(g) & u_{3}(g) & u_{2}(g) \\
0 & 0 & 1 & s_{3}(g) & t_{2}(g) \\
0 & 0 & 0 & 1 & 0 \\
0 & 0 & 0 & 0 & 1
\end{array}\right) \, .  \]
Let us multiply two of these, say~$g$ and~$g'$, using the shorthand~$s_2= s_2(g)$ and~$s_2'=s_2(g')$, and so on: 
\[ gg' = \left(\begin{array}{rrrrr}
1 & 0 & t_{1} + t'_{1} & t_{1} s'_{3} + u_{1} + u'_{1} & t_{1} t'_{2} + z + z' \\
0 & 1 & s_{2} + s'_{2} & s_{2} s'_{3} + u_{3} + u'_{3} & s_{2} t'_{2} + u_{2} + u'_{2} \\
0 & 0 & 1 & s_{3} + s'_{3} & t_{2} + t'_{2} \\
0 & 0 & 0 & 1 & 0 \\
0 & 0 & 0 & 0 & 1
\end{array}\right) \, . \tag{*} \]
We shall use the set-theoretic section~$\sec \colon N \to \cent(S)$ given by 
\[ \sec(g) =  \left(\begin{array}{rrrrr}
1 & 0 & t_{1}(g) & 0 & 0 \\
0 & 1 & s_{2}(g) & 0 & 0 \\
0 & 0 & 1 & s_{3}(g) & t_{2}(g) \\
0 & 0 & 0 & 1 & 0 \\
0 & 0 & 0 & 0 & 1
\end{array}\right) \, . \]

Let us prove (1). Using the section~$N \to \cent(S)/S_1$ induced by~$\sec$, we end up with the bijection of sets $\Phi \colon S/S_1 \times N \to \cent(S)/S_1$ given by 
\[ \Phi(x, g) = x \sec(g) = \left(\begin{array}{rrrrr}
1 & 0 & 0 & 0 & 0 \\
0 & 1 & 0 & x & 0 \\
0 & 0 & 1 & 0 & 0 \\
0 & 0 & 0 & 1 & 0 \\
0 & 0 & 0 & 0 & 1
\end{array}\right) \times \sec(g) = \left(\begin{array}{rrrrr}
1 & 0 & t_{1}(g) & 0 & 0 \\
0 & 1 & s_{2}(g) & x & 0 \\
0 & 0 & 1 & s_{3}(g) & t_{2}(g) \\
0 & 0 & 0 & 1 & 0 \\
0 & 0 & 0 & 0 & 1
\end{array}\right) \, .   \]
Here we have used (*) to perform the calculation. A caveat : in these expressions, we have identified~$S/S_1$ with~$\f_2$ (this can be done uniquely!), so that~$x \in S/S_1$ can be seen as an entry ($0$ or $1$) of a matrix. A second caveat is that the matrix displayed is understood modulo~$S_1$ only.

From the theory of group extensions, we have $\Phi (x, g) \Phi (y, g') = \Phi (x + y + c(g, g'), gg')$, where the expression $c(g,g')$, is what we are after, that is,  it is a two-cocycle representing the cohomology class of the extension under scrutiny. So we compute, from (*), that
\[ \Phi (x, g) \Phi (y, g') = \left(\begin{array}{rrrrr}
1 & 0 & t_{1} + t'_{1} & t_{1} s'_{3} & t_{1} t'_{2} \\
0 & 1 & s_{2} + s'_{2} & s_{2} s'_{3} + x + y & s_{2} t'_{2} \\
0 & 0 & 1 & s_{3} + s'_{3} & t_{2} + t'_{2} \\
0 & 0 & 0 & 1 & 0 \\
0 & 0 & 0 & 0 & 1
\end{array}\right) \, . 
  \]
Another useful computational remark is that, in $S$ identified with the additive group of~$2\times 2$-matrices, we have 
\[ \left(\begin{array}{cc}
  a & b \\
  c & d 
\end{array}\right) = (a+b+c+d) e_1 + (a+b) e_2 + (b+d)e_3 + b e_4 \equiv \left(\begin{array}{cc}
  0 & 0 \\
  a+b+c+d & 0
\end{array}\right) ~\textnormal{mod}~ S_1\, .   \]
Thus the last matrix displayed, viewed in~$\cent(S)/S_1$, is also 
\[  \left(\begin{array}{rrrrr}
1 & 0 & t_{1} + t'_{1} & 0 & 0 \\
0 & 1 & s_{2} + s'_{2} &  x + y + s_{2} s'_{3} + t_{1} s'_{3} +  s_{2} t'_{2} + t_{1} t'_{2} & 0  \\
0 & 0 & 1 & s_{3} + s'_{3} & t_{2} + t'_{2} \\
0 & 0 & 0 & 1 & 0 \\
0 & 0 & 0 & 0 & 1
\end{array}\right) \, .  \]
We conclude that, as predicted, $c(g, g') = s_{2}(g) s_{3}(g') + t_{1}(g) s_{3}(g') +  s_{2}(g) t_2{(g')} + t_{1}(g) t_{2}(g') = x_2(g) x_3(g')$, and~$c$ is indeed the cup-product of~$x_2$ and~$x_3$.

The proofs of (2) and (3) are similar.
\end{proof}

%% \begin{proof}
%% Same as the proof of Lemma~\ref{lem-coho-class-D4}. The lengthy argument is omitted. It is a matter of picking lifts for the generators of~$N$, checking their orders, and whether they commute.
%% \end{proof}

\begin{rmk}
Note that in each case~$\cent(S)/S_i \cong C_2^2 \times D_4$.
\end{rmk}

\section{The fundamental cup-products} \label{sec-cup-products}

In this section, we prove Theorem~\ref{thm-main-four-cup-products-without-tildes}, as stated in the Introduction. We shall see that we are led naturally to another statement first, in which the following notation is used. When~$a \in F$, we put~$F_a = F[X]/(X^2  - a)$. When~$e + fX \in F_a$, we define its {\em norm} to be $\Norm_{F_a/F}(e + fX) = e^2 - a f^2$. When a square root~$\sqrt a$ has been chosen in some field containing~$F$, there is a homomorphism $F_a \to F[\sqrt a]$ mapping~$X$ to~$\sqrt a$. Of course when~$a$ is not a square in~$F$, this map is an isomorphism of extensions of~$F$, and the norm just introduced coincides with the usual norm map between fields. However, it is useful to work with~$F_a$ for those occasions when~$a$ is already a square in~$F$.

We start by the implication $(1) \implies (2)$ from Theorem~\ref{thm-main-four-cup-products-without-tildes}.

\subsection{Proving that the four cup-products vanish}

\begin{thm} \label{thm-U5-implies-rational-point}
Let~$F$ be a field of characteristic not~$2$, let $a,b,c,d \in F^\times$ be given and put $E := F[\sqrt{a},\sqrt{d}]$.
Suppose there exist a continuous homomorphism~$\bar \phi \colon G_F \to \bar\U_5(\f_2)$ such that~$\chi_a = s_1 \circ \bar\phi$, $\chi_b = s_2 \circ \bar\phi$, $\chi_c = s_3 \circ \bar\phi$, and $\chi_d = s_4 \circ \bar\phi$. (In other words, we assume that $\langle a, b, c, d \rangle$ is defined.)
Then there exist~$\tilde B \in F_a$ and~$\tilde C \in F_d$ such that the following hold, where $B$ denotes the image of $\tilde B$ under $F_a \rightarrow F[\sqrt{a}]$ and $C$ denotes the image of $\tilde C$ under $F_d \rightarrow F[\sqrt{d}]$.
\begin{enumerate}
  \item One has $\Norm_{F_a/F}(\tilde B) = b$ modulo squares and $\Norm_{F_d/F}(\tilde C) = c$ modulo squares.
  \item One has $(B, c)_{F[\sqrt a]} = 0$, $(b, C)_{F[\sqrt d]}=0$, and $(b,c)_F= 0$.
  \item There is a class~$u \in \h^2(F, \f_2)$ whose image under the canonical restriction map $\h^2(F, \f_2)  \longrightarrow  \h^2(E, \f_2)$ is~$(B,C)_E$.
  \item Suppose that~$\bar \phi $ can be lifted to~$\phi \colon G_F \to \uf$. (In other words, assume that $\langle a, b, c, d \rangle$ vanishes.) Then one has~$(B,C)_E= 0$.
\end{enumerate}

\end{thm}

Note that the mere existence of~$\tilde B$ and~$\tilde C$ implies also that~$(a,b)_F= 0$ and~$(c,d)_F = 0$, by Proposition~\ref{prop-D4-iff-cup-prop-iff-norm}.

\begin{proof}

  First we consider the composition
\[ \bbar \phi \colon G_F \longrightarrow \bar{U}_5(\f_2) \longrightarrow \bar{\U}_5(\f_2)/S = G = D_4 \times D_4 \, .   \]
Projecting further onto the left factor, we make a first use of Proposition~\ref{prop-D4-iff-cup-prop-iff-norm}.
We draw the existence of~$x, y, z \in F$ satisfying $x^2 - ay^2 = bz^2$, that is~$\Norm_{F_a/F}(\tilde B) = bz^2$, with~$\tilde B = x + y X \in F_a$.
If~$B= x + y \sqrt a$ is the corresponding element, then the Proposition says that we can arrange to have~$\chi_{bB} = t_1 \circ \bar \phi = x_1 \circ \bar \phi$, $\chi_B =(s_2 + t_1) \circ \bar \phi  = x_2 \circ \bar \phi $. (We fully use the notation from~\S\ref{subsec-U5-subgroups}.)

Using the right factor, we find elements~$\tilde C$ and~$C$ similarly, such that~$\chi_C = x_3 \circ \bar \phi$, $\chi_{cC} = x_4 \circ \bar\phi$.
This also proves the first assertion.

We prove that the cup-products vanish as announced in the second assertion, starting with~$(b,c)_F= 0$ (which of course does not depend on the choices for~$B$ and~$C$). For this, consider the map
\[ \uf \longrightarrow \U_3(\f_2) = D_4  \]
which discards the top row and the rightmost column of an element of~$\uf$; this factors through~$\bar U_5(\f_2)$.
Postcomposing~$\bar \phi$ with this, we draw from Proposition~\ref{prop-D4-iff-cup-prop-iff-norm} that~$(b, c)_F= 0$, as required.

Next we turn to the proof of~$(B, c)_{F[\sqrt a]} = 0$, the cup-product~$(b,C)_{F[\sqrt d]}$ being treated in a similar way.
Define a group homomorphism~$\pi\colon \ker(s_1) \subset \bar\U_5(\f_2) \to \U_3(\f_2) = D_4$ by
\[ g \mapsto \left(\begin{array}{ccc}
  1 & t_1(g) & u_1(g) \\
  0 & 1 & s_3(g) \\
  0 & 0 & 1
\end{array}\right) \, .   \]
One checks that~$\pi$ is well-defined.
Note that~$\pi \circ \bar\phi : G_{F[\sqrt{a}]} \rightarrow D_4$ is a lift for~$(t_1 \circ \bar\phi, s_3\circ \bar\phi) \colon G_{F[\sqrt a]} \to \f_2 \times \f_2$ (using that~$G_{F[\sqrt a]} = \bar\phi^{-1}(\ker(s_1))$).
Therefore, we see from Proposition~\ref{prop-D4-iff-cup-prop-iff-norm} that the cup product of~$t_1 \circ \bar\phi$ and~$s_3 \circ \bar\phi$ is zero.
This means, in alternative notation, that~$(Bb, c)_{F[\sqrt a]} = 0$, so~$(B,c)_{F[\sqrt a]}= 0$.

We now turn to the third assertion.
Let~$\alpha $ be the cohomology class of the extension
\[ 0 \longrightarrow \f_2 \longrightarrow \uf \longrightarrow \bar\U_5(\f_2) \longrightarrow 1 \, .   \]
Let us write down an explicit two-cocycle~$\gamma $ representing~$\alpha $.
First, recall the functions $s_1, t_1, \ldots \colon \uf \to \f_2$ introduced in~\S\ref{subsec-notation-Un}.
Multiplying two matrices~$g, h \in \uf$, the top-right coefficient must be~$z(g) + z(h) + \gamma (g, h)$, and we deduce
\[ \gamma (g, h) = s_1(g)u_2(h) + t_1(g)t_2(h) + u_1(g)s_4(h) \, .   \]
To obtain a two-cocycle representing the pull-back~$\bar\phi^*(\alpha ) \in \h^2(F, \f_2)$, we only need compose with~$\bar\phi$. Restricting to the subgroup~$G_E$ where~$s_1 \circ \bar\phi$ and~$s_4 \circ \bar\phi$ both vanish, we obtain that~$\bar\phi^*(\alpha )_E$ is represented by the two-cocyle
\[ \sigma, \tau  \mapsto (t_1\circ \bar\phi(\sigma )) \, (t_2 \circ \bar\phi(\tau )) \, .   \]
It follows that~$\bar\phi^*(\alpha )_E= t_1 \circ \bar\phi \cupp t_2 \circ \bar\phi$, the cup-product of the classes~$t_i \circ \bar\phi \in \h^1(E, \f_2)$. Or in other words $\bar\phi^*(\alpha )_E = (bB, cC)_E$.

Given that~$(B, c)_{F[\sqrt a]} = 0$, $(b, C)_{F[\sqrt d]} = 0$, and~$(b, c)_F = 0$, it follows that~$u= \phi^*(\alpha ) \in \h^2(F, \f_2)$ satisfies~$u_E = (B, C)_E$.

Finally, suppose that~$\phi $ exists as in the fourth assertion.
We note that~$G_E = \phi^{-1}(\cent(S))$ (again the notation $\cent(S)$ for the centralizer of~$S$ is from~\S\ref{subsec-U5-subgroups}, and we had noted~$\cent(S) = \ker(s_1) \cap \ker(s_4)$).
The composition
\[ f  \colon G_E \longrightarrow \cent(S) \longrightarrow N  \]
factors via~$\cent(S)/S_i$ (for~$i=1, 2, 3$), so from the Lemma~\ref{lem-classes-extensions-of-N},  we must have~$f^*(x_2x_3)= 0$, $f^*(x_1x_3)= 0$ and~$f^*(x_2x_4) = 0$.
But these translate as~$(B, C)_E = 0$, which we were after, and $(bB, C)_E = 0$, ~$(B, cC)_E= 0$, consistently with the above.
\end{proof}

\subsection{Shapiro's lemma and the converse}

Let~$G$ be a finite group, let~$N$ be a subgroup, and let~$k$ be a field. For any~$kN$-module~$A$, we let~$\co^G_N (A)$ denote~$\Hom_N(kG, A)$, which is a (left) $kG$-module with action~$(\sigma \cdot f)(x) = f(x \sigma )$. (We are thinking of~$G$ and~$N$ as being the groups bearing those names in the discussion above, with~$k = \f_2$, and~$A$ having trivial action.) The well-known Shapiro's lemma states the existence of an isomorphism
\[ sh \colon \h^2(G, \co_N^G(A)) \longrightarrow \h^2(N, A) \, .  \]
More precisely, the map is obtained using~$\ev \colon \co_N^G(A) \to A$ which evaluates at~$1 \in G$, followed by  restriction (in cohomology) to~$N$. Note, if the class~$\alpha \in \h^2(G, \co_N^G(A))$ describes the extension
\[ 0 \longrightarrow \co_N^G(A) \longrightarrow \Gamma \stackrel{p}{\longrightarrow} G \longrightarrow 1 \, ,  \]
then~$sh(\alpha )$ corresponds to
\[ 0 \longrightarrow A \longrightarrow \frac{p^{-1}(N)} {\ker (\ev)} \longrightarrow N \longrightarrow 1 \, ,   \]
as is easily verified.

In order to recognize that a given $G$-module, say~$S$, is isomorphic to~$\co_N^G(A)$ for some~$A$, let us merely consider the case where~$A=k^r$ with trivial~$N$-action, and assume that~$N$ is normal in~$G$ to boot. Then~$\co_N^G(A) = \left( k[G/N]^*\right)^r$. The dual module~$\left(k[G/N]\right)^*$ is free of rank one, that is, it is isomorphic to~$k[G/N]$. (Consider the map taking~$1 \in G$ to $\delta_1$, the Dirac delta function at~$1$. Thus, $\delta_1(N) = 1$ and $\delta_1(g \cdot N) = 0$ if $g \notin N$.) We conclude that~$S$ is isomorphic to~$\co_N^G(k^r)$ if and only if~$N$ acts trivially and we can find a basis for~$S$ as a free~$k[G/N]$ of rank~$r$. If this basis is~$\varepsilon_1, \ldots, \varepsilon_r$, then~$\ker(ev)$ is the~$k$-vector space spanned by~$g \varepsilon_i$, for~$g \in G$, $g \ne 1$, $i= 1, \ldots, r$.

For example, let~$G, N, S$ recover their concrete meanings as in \S\ref{subsec-U5-subgroups} (all the accompanying notation will be used, too). Then Lemma~\ref{lem-S-is-free} asserts that~$S \cong \co_N^G(\f_2)$, in such a way that~$\ker(ev)$ is identified with~$S_1$. As a result, the cohomology class of
\[ 0 \longrightarrow S \longrightarrow \uf \longrightarrow G \longrightarrow 1 \tag{*} \]
corresponds via~$sh$ to the cohomology class of the first extension treated in Lemma~\ref{lem-classes-extensions-of-N}, that is, $x_2x_3$.

But~$S$ can be regarded in another way.
Consider the subgroup~$G' \subset G$ of elements mapping into~$\langle \sigma_1 \sigma_4\rangle$ under~$G \to G/N$ (equivalently, $g \in G'$ if~$s_1(g) = s_4(g)$), and view~$S$ as a~$G'$-module.
Lemma~\ref{lem-S-is-free} shows that~$S$ is a free~$\f_2[\langle \sigma_1 \sigma_4\rangle] = \f_2[G'/N]$-module, with basis~$e_1, e_2$ (or alternatively~$e_1, e_3$).
Thus we also have~$S \cong \co_N^{G'}(\f_2 \oplus \f_2)$, and~$\ker(ev)$ is spanned by~$e_3, e_4$ (or~$e_2, e_4$ in the alternative).
Now the cohomology class of
\[ 0 \longrightarrow S \longrightarrow \uf' \longrightarrow G' \longrightarrow 1 \, ,   \]
where~$\uf'$ is the preimage of~$G'$, is taken by Shapiro to that of the extension
\[ 0 \longrightarrow \f_2 e_1 \oplus \f_2 e_2 \longrightarrow \frac{\cent(S)} {\langle e_3, e_4 \rangle} \longrightarrow N \longrightarrow 1 \, .  \]
This extension is described by two classes in~$\h^2(N, \f_2)$, corresponding to the exact sequences obtained by factoring out~$e_1$ and~$e_2$ respectively. From Lemma~\ref{lem-classes-extensions-of-N}, these are~$x_1x_3$ and~$x_2x_3$ respectively. With the alternative choice of basis for~$S$, this discussion ends with~$x_2x_4$ and~$x_2x_3$.

There is a well-known version of Shapiro's lemma for profinite groups (cf. \cite[Ch. 1 \S 6]{MR2392026}), which can be deduced from the version mentioned above using a straightforward limit argument.
We record this version below in the context of Galois cohomology, since we will use it later on.

\begin{lem}
Let~$E/F$ be a finite Galois extension, let~$A$ be a trivial, discrete~$G_E$-module, and consider~$\co_{G_E}^{G_F}(A)$, a discrete~$G_F$-module. Then Shapiro's map
\[ \h^2(F, \co_{G_E}^{G_F}(A)) \longrightarrow \h^2(E, A) \, ,  \]
defined as above, is an isomorphism. \hfill \qed
\end{lem}

With the preparations above, we can now prove our primary converse to Theorem \ref{thm-U5-implies-rational-point}.
\begin{thm}\label{thm-rational-point-implies-U5}
Let~$F$ be a field of characteristic not~$2$, let~$a, b, c, d \in F^\times$ be given, and put~$E := F[\sqrt a, \sqrt d]$.
Assume that there exist~$\tilde B \in F_a$ such that~$\Norm_{F_a/F}(\tilde B)= b$ modulo squares, and~$\tilde C \in F_d$ such that $\Norm_{F_d/F}(\tilde C) = c$ modulo squares, with the following additional property: if~$B$ is the image of~$\tilde B$ under~$F_a \to F[\sqrt a]$ and~$C$ is the image of~$\tilde C$ under~$F_d \to F[\sqrt d]$, then $(B,C)_E = (B, c)_E =(b, C)_E = (b,c)_E= 0$.

 Then there exist a continuous homomorphism~$\phi : G_F \to \uf$ such that~$s_1 \circ \phi = \chi_a$, $s_2 \circ \phi= \chi_b$, $s_3 \circ \phi = \chi_c$ and~$s_4 \circ \phi = \chi_d$.
 In other words, the Massey product~$\langle a, b, c, d\rangle$ is defined and vanishes.
\end{thm}

\begin{proof}
We start by assuming that neither~$a$ nor~$d$ is a square in~$F$ (we identify~$F_a$ and~$F_d$ with~$F[\sqrt a]$ and~$F[\sqrt d]$ respectively). From Proposition~\ref{prop-D4-iff-cup-prop-iff-norm} (applied twice), we obtain a homomorphism
\[ f \colon G_F \longrightarrow D_4 \times D_4 = G  \]
such that~$s_i \circ f= \chi_a, \chi_b, \chi_c, \chi_d$ according as~$i= 1, 2, 3, 4$, and also~$f^*(x_1) = \chi_{bB}$, $f^*(x_2) = \chi_B$, $f^*(x_3) = \chi_C$, $f^*(x_4)= \chi_{cC}$. If~$\alpha \in \h^2(D_4 \times D_4, S)$ is the class of the extension
\[ 0 \longrightarrow S \longrightarrow \uf \longrightarrow D_4 \times D_4 \longrightarrow 1 \, ,   \]
then its pull-back~$f^*(\alpha) \in \h^2(F,S)$ is represented by the fibered-product of~$\uf$ with~$G_F$ over~$D_4 \times D_4$ (via $f$). 
To conclude the proof of the theorem, we will show that one has~$f^*(\alpha) = 0$, so that this fibered-product is a split extension of $G_F$ by $S$.
The composition of such a splitting with the projection to $\uf$ provides the necessary homomorphism $\phi$.

Note that~$G_E = f^{-1}(N)$. Suppose first that~$[E : F]= 4$. The~$G_F$-module~$S$ is isomorphic to~$\co_{G_E}^{G_F}(\f_2)$, the trivial module of~$G_E$ (co)induced to~$G_F$. Shapiro's lemma applies, yielding the isomorphism
\[ \h^2(F, S) \longrightarrow \h^2(E, \f_2) \, .   \]
From the comments above and the naturality of Shapiro's isomorphism, we see that~$f^*(\alpha )$ is taken to~$f^*(x_2x_3) = f^*(x_2)f^*(x_3) = (B,C)_E = 0$, so we are done in this case.

We turn to the case where $a = d$ mod squares and $[E:F] = 2$; another way to phrase this is by saying that the image of~$f$ lies within~$G'$. Now the~$G_F$-module~$S$ is isomorphic, although not canonically, to~$\co_{G_E}^{G_F}(\f_2 \oplus \f_2)$: we have at least the two possibilities given in the discussion preceding the proof, and for definiteness say we pick the basis~$e_1, e_2$.

Now Shapiro's lemma gives an isomorphism
\[ \h^2(F, S) \longrightarrow \h^2(E, \f_2 \oplus \f_2) = \h^2(E, \f_2) \oplus \h^2(E, \f_2) \, .   \]
This takes~$f^*(\alpha )$ to a pair of cohomology classes, and again from naturality, they are~$f^*(x_2x_3) = (B,C)_E$ and~$f^*(x_1 x_3) = (bB, C)_E$. These are both zero by assumption, and~$f^*(\alpha ) = 0$ also in this case.

Finally, suppose that~$a$ is a square in~$F$ (a symmetric argument deals with the case when~$d$ is a square). A neat way to handle this is to use the Massey Vanishing Conjecture for~$n=3$, (which is now a theorem, see~\cite{Matzri} \cite{EM-triple} \cite{MT} \cite{MT2}).
First we claim that~$(b,c)_F= 0$. Indeed, $(b,c)_E = 0$ by assumption; if~$d$ is also a square in~$F$ then~$E=F$, and otherwise~$E=F[\sqrt d]$ and~$\Norm_{E/F}(C)=c$ mod squares, so~$(b, C)_E=0$ implies~$(b,c)_F = 0$ by the projection formula in group cohomology. We also have~$(c, d)_F=0$ by Proposition~\ref{prop-D4-iff-cup-prop-iff-norm}.  We deduce that the Massey product~$\langle b, c, d\rangle$ vanishes, and so that there is a continuous~$\psi \colon G_F \to \U_4(\f_2)$ compatible with~$b, c, d$. However, if we see~$\U_4(\f_2)$ as the subgroup of~$\uf$ consisting of those matrices whose first row is that of the identity matrix, we see that we have built~$G_F \to \uf$ compatible with~$1, b, c, d$, as required.
\end{proof}

\begin{rmk} \label{rmk-three-vanishing-not-needed}
It is possible to avoid relying on the Massey Vanishing Conjecture for~$n=3$ if one wishes to do so. Here is a sketch. When~$d$ is not a square, use another argument based on Shapiro's lemma, this time replacing~$G'$ by~$G''$, the subgroup of~$G$ of elements mapping to~$\langle \sigma_d \rangle \subset G/N$. When~$d$ is a square, define~$f$ as in the first part of the proof; this time~$f$ takes values in~$N$. The exact sequence
\[ 0 \longrightarrow S \longrightarrow \cent(S) \longrightarrow N \longrightarrow 1  \]
is {\em central}, and controlled by four cohomology classes in~$\h^2(N, \f_2)$, pulling back to~$(B,C)$, $(b, C)$, $(B,c)$ and~$(b,c)$ in~$\h^2(F, \f_2)$ under~$f^*$.
\end{rmk}

The proof of Theorem~\ref{thm-main-four-cup-products-without-tildes} is now almost complete. In fact, what we have is this:

\begin{thm} \label{thm-main-four-cup-products}
  Let~$F$ be a field of characteristic not~$2$, and let~$a, b, c, d \in F^\times$ be given.
  Then the following statements are equivalent.

  \begin{enumerate}
  \item $\langle a, b, c, d \rangle$ is defined and vanishes.
    \item There exist~$\tilde B \in F_a$ such that~$\Norm_{F_a/F}(\tilde B)= b$ modulo squares, and~$\tilde C \in F_d$ such that $\Norm_{F_d/F}(\tilde C) = c$ modulo squares, with the following extra property. If~$B$ is the image of~$\tilde B$ under~$F_a \to F[\sqrt a]$ and~$C$ is the image of~$\tilde C$ under~$F_d \to F[\sqrt d]$, then $(B,C)_{F[\sqrt a, \sqrt d]} = 0$, $(B, c)_{F[\sqrt a]} = 0$, $(b, C)_{F[\sqrt d]}=0$, and $(b,c)_F= 0$.
    \item There exist~$\tilde B \in F_a$ such that~$\Norm_{F_a/F}(\tilde B)= b$ modulo squares, and~$\tilde C \in F_d$ such that $\Norm_{F_d/F}(\tilde C) = c$ modulo squares, with the following extra property. If~$B$ is the image of~$\tilde B$ under~$F_a \to F[\sqrt a]$ and~$C$ is the image of~$\tilde C$ under~$F_d \to F[\sqrt d]$, then $(B,C)_E = (B, c)_E =(b, C)_E  = (b,c)_E= 0$.

  \end{enumerate}
\end{thm}

\begin{proof}
Theorem~\ref{thm-U5-implies-rational-point} shows the implication (1) $\implies$ (2), while (2) $\implies$ (3) is trivial, and Theorem~\ref{thm-rational-point-implies-U5} shows (3) $\implies$ (1).
\end{proof}

The only difference with Theorem~\ref{thm-main-four-cup-products-without-tildes} is the presence of the elements~$\tilde B$ and~$\tilde C$ instead of just~$B, C$.
Clearly, if neither~$a$ nor~$d$ is a square in~$F$, then the two results are the same.
On the other hand, when one of these two elements is a square, in fact when one of~$a, b, c, d$ is a square, things become very easy, as we proceed to show in the following subsection.

\subsection{Some trivial cases}

\begin{lem} \label{lem-when-abcd-square}
Let~$F$ be a field of characteristic different from~$2$ and let~$a, b, c, d \in F^\times$ be given.
Suppose that $(a, b)_F = (b,c)_F = (c, d)_F = 0$.
Finally, assume that one of~$a, b, c, d$ is a square in~$F$. Then the Massey product $\langle a, b, c, d \rangle$ is defined and vanishes.
\end{lem}

\begin{proof}
  First assume that~$a$ is a square, or equivalently~$a=1$. Then the argument given in the last paragraph of the proof of Theorem~\ref{thm-rational-point-implies-U5} shows that $0 \in \langle a, b, c, d \rangle$.  The case when~$d=1$ is treated similarly.

  On the other hand, suppose that~$b= 1$. From~$(c, d)_F = 0$, we draw the existence of $G_F \to \U_3(\f_2)$ compatible with~$c$ and~$d$, by Proposition~\ref{prop-D4-iff-cup-prop-iff-norm}. The element~$a$ itself defines~$\chi_a \colon G_F \to \f_2 \cong  C_2$. Since the subgroup of~$\uf$ generated by~$\sigma_1, \sigma_3, \sigma_4$ is isomorphic to~$C_2 \times \U_3(\f_2)$, we see immediately that we may combine our two homomorphisms into one of the form~$G_F \to \uf$, showing that $0 \in \langle a, 1, b, c \rangle$. The case when~$c=1$ is treated similarly.
\end{proof}

\begin{rmk}
The above proof, {\em via} the reference to the proof of Theorem~\ref{thm-rational-point-implies-U5}, uses the Massey Vanishing Conjecture for~$n=3$. Without this, it is still a general fact about Massey products that $\langle a, b, c, d \rangle$ vanishes when it is defined and one of them is a square (\cite{fenn}, Lemma 6.2.4). However, the statement just given is stronger.
\end{rmk}

We proceed to show how we can improve the statement of Theorem~\ref{thm-main-four-cup-products} to that of Theorem~\ref{thm-main-four-cup-products-without-tildes} from the Introduction, so that the two are in fact equivalent. We have already mentioned that this is obvious when neither~$a$ nor~$d$ is a square in~$F$.

Let us call (1A), (2A), (3A) the conditions of Theorem~\ref{thm-main-four-cup-products-without-tildes}, and keep (1), (2), (3) for those of Theorem~\ref{thm-main-four-cup-products}, which we know are equivalent.  Note (1) $=$ (1A).

Suppose~$a$ is a square in~$F$, but not~$d$. Assume condition (1A). We must show that condition (2A) holds. Indeed, from condition (2), we have the element~$C$ with~$\Norm_{F[\sqrt d]/F}(C)=c$ mod squares, and satisfying~$(b, C)_{F[\sqrt d]} = 0$, and moreover~$(b, c)_F= 0$. Now put~$B= b \in F=F[\sqrt a]$, so that~$\Norm_{F[\sqrt a]/F}(B) = B = b$. Then~$(B,C)_E = (b, C)_E = 0$, while~$(B,c)_{F[\sqrt a]} = (b,c)_F = 0$. We do have condition (2A), and so also (3A).

Conversely, condition (3A) is enough to ensure that~$(a,b)_F=(1,b)_F= 0$, $(b,c)_F=0$ (apply the projection formula to~$(b,C)_E = 0$) and~$(c, d)_F=0$ (merely because~$C$ exists, cf Proposition~\ref{prop-D4-iff-cup-prop-iff-norm}). So the last Lemma applies and shows that condition (1A) holds.

The situation when~$a$ is not a square, but~$d$ is, is clearly similar.

Now suppose~$a$ and~$d$ are both squares. Suppose condition (1A) holds, and so also (2), and we have~$(a, b)=(c,d)=0$ (because~$\tilde B$ and~$\tilde C$ exist, cf remark after Theorem~\ref{thm-U5-implies-rational-point}) and~$(b,c)=0$. Thus condition (2A) holds with~$B=b$ and~$C=c$, and (2A) $=$ (3A) here. Conversely, condition (3A) contains the statement $(b,c)=0$, while $(a, b)= (1, b)=0$ and~$(c, d)= (c, 1)= 0$, and we see by the last Lemma that condition (1A) holds.
We have therefore just proven Theorem \ref{thm-main-four-cup-products-without-tildes}.

\begin{rmk}
Theorem~\ref{thm-main-four-cup-products} is heavier on notation than Theorem~\ref{thm-main-four-cup-products-without-tildes}, so we have deemed it unfit for the Introduction. However, the extra case-by-case considerations needed to establish the latter, as just given, are perhaps an indication that it is less natural. (Also, it relies more seriously on the Massey Vanishing Conjecture for~$n=3$, see Remark~\ref{rmk-three-vanishing-not-needed}.) In the sequel we shall refer to Theorem~\ref{thm-main-four-cup-products}, rather than to Theorem~\ref{thm-main-four-cup-products-without-tildes}.
\end{rmk}

\subsection{Maps from profinite groups into~$\uf$}

A routine modification of the arguments given above produces the next result.

\begin{thm}
  Let~$\Gamma $ be a profinite group, and let~$\chi_1, \chi_2, \chi_3, \chi_4 \in \h^1(\Gamma, \f_2)$ be given.
  Put~$\Gamma_1 := \ker(\chi_1)$, $\Gamma_4 := \ker(\chi_4)$, and $\Gamma_{14} := \Gamma_1 \cap \Gamma_4$. The the following statements are equivalent.

  \begin{enumerate}
  \item There exists a continuous homomorphism $\phi \colon \Gamma \to \uf$ such that~$\chi_i = s_i \circ \phi $ for~$i= 1, 2, 3, 4$. In other words, $\langle \chi_1, \chi_2, \chi_3, \chi_4 \rangle$ is defined and vanishes.
    \item There exist a continuous homomorphism~$\Gamma \to D_4$ given by
\[ \gamma \mapsto \left(\begin{array}{rrr}
  1 & \chi_1(\gamma ) & \zeta_1(\gamma ) \\
  0 & 1 & \chi_2(\gamma ) \\
  0 & 0 & 1
\end{array}\right) \, ,   \]
and another one given by
\[ \gamma \mapsto \left(\begin{array}{rrr}
  1 & \chi_3(\gamma ) & \zeta_2(\gamma ) \\
  0 & 1 & \chi_4(\gamma ) \\
  0 & 0 & 1
\end{array}\right) \, ,   \]
such that~$\zeta_1 |_{\Gamma_{14}} \cupp \zeta_2 |_{\Gamma_{14}} = 0$, $\zeta_1 |_{\Gamma_1} \cupp \chi_3|_{\Gamma_1} = 0$, $\chi_2 |_{\Gamma_4} \cupp \zeta_2 |_{\Gamma_{4}} = 0$, $\chi_2 \cupp \chi_3 = 0$.

\item $\zeta_1$ and~$\zeta_2$ exist as above and satisfy~$\zeta_1 |_{\Gamma_{14}} \cupp \zeta_2 |_{\Gamma_{14}} = \zeta_1 |_{\Gamma_{14}} \cupp \chi_3|_{\Gamma_{14}} = \chi_2 |_{\Gamma_{14}} \cupp \zeta_2 |_{\Gamma_{14}} = \chi_2|_{\Gamma_{14}} \cupp \chi_3|_{\Gamma_{14}} = 0$.
\end{enumerate}
\end{thm}

We shall have no use for this Theorem in the sequel, so we leave the proof to the reader.

\section{First applications}
\label{sec-first-applications}

While our main objective in this paper is to prove the $4$-Massey Vanishing Conjecture for number fields, there are a few cases which can be treated {\em over any field}. Usually, we use the following trick when working ``by hand''. It will have a more theoretical use below, too.

\begin{prop} \label{prop-modifiers}
  Let~$F$ be a field of characteristic not~$2$, and let~$a, b, c, d \in F^\times$. Let~$\tilde B_0 \in F_a$ be any initial element such that~$\Norm_{F_a/F}(\tilde B_0) = b$, and likewise let~$\tilde C_0 \in F_d$ be any initial element such that~$\Norm_{F_d/F}(\tilde C_0)= c$. The following statements are equivalent:

  \begin{enumerate}
  \item There exist~$\tilde B \in F_a$ with~$\Norm_{F_a/F}(\tilde B)=b z_1^2$, and~$\tilde C \in F_d$ with~$\Norm_{F_d/F}(\tilde C)= cz_2^2$, for some~$z_1, z_2 \in F^\times$, such that we have simultaneously $(B, C)_{F[\sqrt a, \sqrt d]}=0$, $ (B, c)_{F[\sqrt a]} =0$,  $(b, C)_{F[\sqrt d]}=0$,  $ (b, c)_F = 0$, where~$B \in F[\sqrt a]$ and~$C \in F[\sqrt d]$ correspond to~$\tilde B, \tilde C$ respectively.

    \item There exist~$\beta , \gamma \in F^\times$ such that~$\tilde B = \beta \tilde B_0$ and~$\tilde C = \gamma \tilde C_0$ satisfy the previous condition.
\end{enumerate}

\end{prop}

Of course, by Theorem~\ref{thm-main-four-cup-products} these conditions are also equivalent to the vanishing of~$\langle a, b, c, d \rangle $, but the point of the Proposition is to show that it makes sense to start with any~$\tilde B_0, \tilde C_0$ and then look for the ``modifiers'' $\beta, \gamma $. 

\begin{proof}
It is plain that (2) implies (1), so here is the non-trivial part. Assume that~$\tilde B, \tilde C$ exist as in the first part. We claim that we can in fact find~$\beta, \gamma \in F$ and~$\lambda \in F[\sqrt a]$, $\mu \in F[\sqrt d]$ such that~$B = \beta B_0 \lambda ^{-2}$ and~$C = \gamma C_0 \mu^{-2}$, where~$B_0 \in F[\sqrt a]$ and~$C_0 \in F[\sqrt d]$ correspond to~$\tilde B_0, \tilde C_0$ respectively. Clearly this will give the result.

To prove the claim, we suppose first that~$a$ is not a square in the field~$F$, and write $\Norm_{F[\sqrt a]/F} (Bz_1^{-1}) = b = \Norm_{F[\sqrt a]/F}(B_0)$. Since we have $\Norm_{F[\sqrt a]/F}(Bz_1^{-1}B_0^{-1}) = 1$, by Hilbert 90 we have
\[ \frac{Bz_1^{-1}} {B_0} = \frac{\sigma (\lambda )} {\lambda }  \]
for some~$\lambda \in F[\sqrt a]$, where~$\sigma $ is the non-trivial element of~$\gal(F[\sqrt a]/ F)$. Rewrite this~$B = \beta B_0 \lambda ^{-2}$, with~$\beta = z_1 \lambda \sigma (\lambda ) \in F$, and we are done.

If on the other hand~$a$ is a square in~$F$, then~$B_0, B \in F^\times$, so we may put~$\beta = \frac{B} {B_0}$ and~$\lambda = 1$.

A similar argument works with~$C, C_0$.
\end{proof}

We give a series of examples. These will demonstrate how, in practice, one looks for~$\beta $ and~$\gamma $ rather than~$B$ and~$C$ directly. Logically speaking, only the trivial implication of the last Proposition is used, at this point (although knowing that the converse holds gives us confidence in the whole approach).

The elements~$a$ and~$d$ are always assumed not to be squares in~$F$, so we identify~$F_a$ and~$F[\sqrt a]$, as well as~$F_d$ and~$F[\sqrt d]$, and~$\tilde B_0 = B_0$, $\tilde C_0 = C_0$. As you have guessed, the characteristic is always~$\ne 2$. We will make some forward references to the results of the next section, which reduce the number of conditions to check.

\begin{ex}[The abaa case] \label{ex-abaa}  We show that when~$(a, b) = (a, a) = 0$, the Massey product~$\langle a, b, a, a \rangle$ is defined and vanishes, under the assumption that~$a \ne b$ modulo squares, and that~$a$ and~$b$ are not squares.

  Clearly we can find~$B_0$ and~$C_0$ as in the Proposition, by our assumption (and using Proposition~\ref{prop-D4-iff-cup-prop-iff-norm}, of course). We claim that~$B_0$ and~$C_0$ form a basis for~$F[\sqrt a]$ as an~$F$-vector space. Indeed, if we had~$\beta B_0 = \gamma C_0$ for~$\beta, \gamma \in F^\times$, then by taking norms to~$F$ we would find that~$b=a$ modulo squares, a contradiction. Therefore there must exist~$\beta, \gamma \in F$ such that
\[ \beta B_0 + \gamma C_0 = 1 \, .   \]
This implies~$(\beta B_0, \gamma C_0)_{F[\sqrt a]} = 0$. Moreover~$(\beta B_0, a)_{F[\sqrt a]} = (\beta B_0, 1)_{F[\sqrt a]} = 0$. By Lemma~\ref{lem-eq-scrE-iff-cup-product-a=d} below, the other cup-products vanish automatically (the reader will also enjoy looking for a direct argument). By Theorem~\ref{thm-main-four-cup-products}, the Massey product is defined and vanishes.
\end{ex}

\begin{ex}[The aaaa case] \label{ex-aaaa}
  Now we complete the discussion of the previous example, and turn to the case when~$a=b$ modulo squares (still assuming that~$a$ is not a square). We show that~$(a, a) = 0$ implies that $\langle a, a, a, a \rangle$ is defined and vanishes.

  Since~$(a, a)= 0$, we draw the existence of~$B_0$ with~$\Norm_{F[\sqrt a]/F}(B_0)= a$ from Proposition~\ref{prop-D4-iff-cup-prop-iff-norm}. It is always true that~$(a, -a) = 0$, so we have~$(a, -1) = 0$, implying that~$a$ is a sum of two squares in~$F$. We invoke the ``Norm Principle'' from~\cite{elmanlam}: from the fact that~$\Norm_{F[\sqrt a]/F}(B_0)$ is the product of two elements, namely~$a$ and~$1$, which are each the sum of two squares in~$F$, this principle implies that there exists~$\beta \in F^\times$ such that~$B= \beta B_0$ is the sum of two squares in~$F[\sqrt a]$. In turn, this means that~$(B, -1) = 0$ and so~$(B, B)= 0$.

  Obviously~$(B, a) = (B, 1) = 0$ in the cohomology of~$F[\sqrt a]$ since~$a$ is a square there, so Theorem~\ref{thm-main-four-cup-products} applies (with~$C=B$).

\end{ex}

\begin{ex}[The abad case] \label{ex-abad}
  One more example in the same style. We prove that~$\langle a, b, a, d \rangle$ is defined and vanishes, under the following assumptions: none of $a, b, d \in F^\times$ is a square, neither is~$ad$, and~$(a, b) = (a, d)= 0$.

  Pick~$B_0$ such that~$\Norm_{F[\sqrt a]/F} (B_0)= b$ and~$C_0$ such that~$\Norm_{F[\sqrt d]/F}(C_0) = a$. Put~$E = F[\sqrt a, \sqrt d]$. From the fact that~$\Norm_{E/F[\sqrt a]}(C / \sqrt a) = 1$, we draw {\em via} Hilbert 90 the existence of~$x \in E^\times$ such that~$Cx^2 = \Norm_{E/F[\sqrt a]}(x) \sqrt a$, and so in particular~$Cx^2 \in F[\sqrt a]$.

  Two cases can occur. First, assume that~$B_0$ and~$C_0x^2$ are linearly independent in~$F[\sqrt a]$ over~$F$. Then there exist~$\beta , \gamma \in F$ such that
\[ \beta B_0 + \gamma C_0 x^2 = 1 \, .   \]
We can see easily that~$\beta $ and~$\gamma $ are both non-zero, for supposing otherwise would lead us to conclude, upon taking norms, that~$b$ or~$a$ is a square in~$F$. Thus $$(\beta B_0, \gamma C_0 x^2)_E = (\beta B_0, \gamma C_0)_E = 0 \, . $$

Let us check that we can come to the same conclusion if we assume, alternatively, that~$B_0 = h C_0x^2$ for some~$h \in F^\times$. Indeed, put~$\beta = 1$ and~$\gamma = -h$ and we do have
\[ (\beta B_0 , \gamma C_0) = (B_0, -hC_0) = (B_0, -hC_0 x^2) = (B_0, -B_0) = 0 \, .   \]

In either case, if we put~$B= \beta B_0$ and~$C = \gamma C_0$, then~$(B, C)_E = 0$. The other hypotheses required to apply Theorem~\ref{thm-main-four-cup-products} are redundant here, from Lemma~\ref{lem-redundancies} below (essentially because of the projection formula).
\end{ex}

\section{Splitting varieties} \label{sec-splitting}

\subsection{The fundamental equation} \label{subsec-reformulation}

Proposition~\ref{prop-D4-iff-cup-prop-iff-norm} gives a necessary and sufficient condition for a cup-product to vanish in Galois cohomology, in terms of a simple polynomial equation. We shall see that the four cup-products of Theorem~\ref{thm-main-four-cup-products} are likewise controlled by a single polynomial equation, taking place in a certain finite étale $F$-algebra, namely~
\[ \E= F[X, Y]/(X^2 - a, Y^2 - d) \cong F_a \otimes_F F_d. \]
More precisely, we establish the following.

\begin{prop} \label{prop-equation-iff-four-cup-products}
Let~$a, b, c, d \in F^\times$. Let~$\tilde B \in F_a \subset \E$ satisfy~$\Norm_{F_a/F}(\tilde B)=b z_1^2$, and let~$\tilde C \in F_d \subset \E$ satisfy~$\Norm_{F_d/F}(\tilde C)= cz_2^2$, for some~$z_1, z_2 \in F^\times$. Let~$B \in F[\sqrt a]$ and~$C \in F[\sqrt d]$ be the corresponding elements. Then the equation
\[ u^2 - \tilde B v^2 = \tilde C  \]
has a solution with~$u, v \in \E$ if and only if we have simultaneously~$(B, C)_{F[\sqrt a, \sqrt d]}=0$, $ (B, c)_{F[\sqrt a]} =0$,  $(b, C)_{F[\sqrt d]}=0$,  $ (b, c)_F = 0$. Alternatively, the same holds with the equation
\[ u^2 \tilde B + v^2 \tilde C = 1 \, .   \]
\end{prop}

The proof will be done in a case-by-case manner (each time getting a slightly more precise statement than that in the Proposition). A quick remark about the equivalence of the two equations, though: it is not a completely general fact, as for example the equation $u^2  + 2v^2  = 1$ has four solutions over~$\z/4\z$ while~$u^2 - v^2 = 2$ has no solution over the same ring. When working over a field of characteristic different from~$ 2$ however, the two problems are equivalent, as elementary computations reveal; since~$\E$ is a direct sum of such fields, we may indeed use either equation. The second is symmetric in~$\tilde B$ and~$\tilde C$, and will be used in the sequel, but the proofs in the remainder of this section will be dealing with the first.

In the ``generic case'' first, that is when~$a$ and~$d$ are linearly independent in~$\li F$, we can and we do identify~$\E$ with~$E= F[\sqrt a, \sqrt d]$. Moreover, we have the following simple situation.

\begin{lem} \label{lem-redundancies}
  Let~$\tilde B, B, \tilde C, C$ be as in the Proposition. Suppose that   $(B,C)_E = 0$. \pierre{right?}

  \begin{enumerate}
  \item[(i)] If neither~$d$ nor~$ad$ is a square in~$F$, then we have~$(B,c)_{F[\sqrt a]} = 0$. If neither~$a$ nor~$ad$ is a square in~$F$, we have~$(b, C)_{F[\sqrt d]} = 0$.
  \item[(ii)] If $(B,c)_{F[\sqrt a]} = 0$ and~$a$ is not a square in~$F$, we have~$(b, c)_F = 0$. Likewise, if $(b, C)_{F[\sqrt d]} = 0$ and~$d$ is not a square in~$F$, we also conclude that~$(b,c)_F = 0$.
  \end{enumerate}

  In particular, when~$[E:F]=4$, the four cup-products from Theorem~\ref{thm-main-four-cup-products} vanish precisely when there are~$u, v \in E$ satisfying
\[ u^2 - v^2 B = C \, .   \]
\end{lem}

\begin{proof}
  If we suppose that~$a$ is not a square in~$F$, then $\Norm_{F[\sqrt a]/F}(B)= bz_1^2$; if~$ad$ is not a square either, than we can also write~$\Norm_{E/F[\sqrt d]}(B) = bz_1^2$. Thus we can use the projection formula from group cohomology, asserting that $$\cores( (B, C)_{E} ) = (b, C)_{F[\sqrt d]} \, , $$ where~$\cores \colon \h^2(E, \f_2) \to \h^2(F[\sqrt d], \f_2)$ is the corestriction. Thus $(B,C)_E = 0$ does imply~$(b, C)_{F[\sqrt d]} = 0$. The other case is treated similarly.

  One also proves (ii) using the projection formula.

For the last statement, we invoke Proposition~\ref{prop-D4-iff-cup-prop-iff-norm} which states that the proposed equation has a solution if and only if~$(B,C)_E= 0$. By the first part, this implies that the other three cup-products also vanish.
\end{proof}

Now suppose that~$a$ and~$d$ are both squares in~$F$. Then we have four~$F$-homomorphisms $p_i \colon \E \to F$, for~$i= 1,2,3, 4$, mapping~$X$ to~$\pm \sqrt a$ and~$Y$ to~$\pm \sqrt d$. Together they induce an isomorphism~$\E \cong F \times F \times F \times F$, by the Chinese Remainder Theorem. (Note that the map~$F \to \E$ which turns~$\E$ naturally into an~$F$-algebra is the diagonal embedding of~$F$ in~$F^4$, under this isomorphism.)

\begin{lem} \label{lem-eq-scrE-iff-cup-product-a-d-squares}
  Suppose~$a$ and~$d$ are both squares in~$F$. Let~$\tilde B, B, \tilde C, C$ be as in the Proposition. Then
\[ u^2 - \tilde B v^2 = \tilde C  \]
has a solution with~$u, v \in \E$ if and only if we have simultaneously~$(B, C)_F = (B, c)_F = (b, C)_F = (b, c)_F = 0$.
\end{lem}

\begin{proof}
  Applying~$p_i$ with~$i= 1, 2, 3, 4$, the equation becomes equivalent to the four equations
\[ u_i^2 - p_i(\tilde B) v_i^2  = p_i(\tilde C)  \]
with unknowns~$u_i, v_i \in F$. This is possible if and only if~$(p_i(\tilde B), p_i(\tilde C)) = 0$ for~$i= 1, 2, 3, 4$. (Note that the calculations to follow will establish that~$p_i(B)\ne 0$, $p_i(C) \ne 0$ for all~$i$, so that the cup-products make sense.)

We need some notation. Let~$B= x + y \sqrt a$ and~$C= x' + y'\sqrt d$, and introduce~$B' = x - y \sqrt a$ and~$C' = x' - y' \sqrt d$, so that~$BB'=bz_1^2\ne 0$, and likewise~$CC' = cz_2^2\ne 0$. To fix our ideas, we assume that the numbering of the homomorphisms~$p_i$ has been made such that
\[ p_1(\tilde B) = p_3(\tilde B) = B \, , \qquad p_2(\tilde B) = p_4(\tilde B) = B' \, ,   \]
\[ p_1(\tilde C) = p_2(\tilde C) = C \, , \qquad p_3(\tilde C) = p_4(\tilde C) = C' \, .   \]
The four cup-products we consider are then~$(B, C)$, $(B', C)$, $(B, C')$ and~$(B', C')$ for~$i= 1, 2, 3, 4$ respectively, all in the cohomology of~$F$. However
\[ (b, C ) = (BB', C) = (B, C) + (B', C) = 0 \, ,   \]
and similarly we draw~$(B, c) = 0$ and~$(b, c)= 0$. One can also work backwards, clearly.
\end{proof}

When~$a$ is a square in~$F$, but~$d$ is not, we view~$\E$ as~$F[\sqrt d][X]/(X^2 - a) \cong F[\sqrt d] \times F[\sqrt d]$ with its two~$F[\sqrt d]$-homomorphisms~$p_1, p_2 \colon \E \to F[\sqrt d]$. The elements~$Y$ and~$\sqrt d$ are identified. A reasoning similar to the above yields:

\begin{lem} \label{lem-eq-scrE-iff-cup-product-a-square}
Suppose that~$a$ is a square in~$F$, and that~$d$ is not. Let~$\tilde B, B, \tilde C, C$ be as in the Proposition. Then the equation
\[ u^2 - \tilde B v^2 = \tilde C  \]
has a solution with~$u, v \in \E$ if and only if we have simultaneously~$(B, C)_{F[\sqrt d]} =  (b, C)_{F[\sqrt d]} = 0$. When this is the case, we have automatically~$(B, c)_F= (b, c)_F = 0$ from Lemma~\ref{lem-redundancies}. \hfill $\square$
\end{lem}

Of course a similar result holds with the roles of~$a$ and~$d$ exchanged.

Finally we turn to the case when neither~$a$ nor~$d$ is a square in~$F$, but~$ad$ is. This is in fact similar to the previous case, except that we now have a choice. Namely, we can produce two~$F[\sqrt d]$-homomorphisms~$p_1, p_2 \colon \E \to F[\sqrt d]$ giving an isomorphism~$\E \cong F[\sqrt d] \times F[\sqrt d]$, identifying~$\sqrt d$ with~$Y$ all along; or, we can alternatively find two~$F[\sqrt a]$-homomorphisms~$p_1', p_2' \colon \E \to F[\sqrt a]$, giving an isomorphism~$\E \cong F[\sqrt a] \times F[\sqrt a]$, identifying~$X$ with~$\sqrt a$ all along. These two isomorphisms are distinct, even if~$a=d$. Using one and then the other, we get:

\begin{lem} \label{lem-eq-scrE-iff-cup-product-a=d}
  Suppose that neither~$a$ nor~$d$ is a square in~$F$, but that~$ad$ is. Let~$\tilde B, B, \tilde C, C$ be as in the Proposition. Then the equation
\[ u^2 - \tilde B v^2 = \tilde C  \]
has a solution with~$u, v \in \E$ if and only if we have simultaneously~$(B, C)_{F[\sqrt d]} =  (b, C)_{F[\sqrt d]} = 0$, if and only if we have simultaneously~$(B, C)_{F[\sqrt a]} =  (B, c)_{F[\sqrt a]} = 0$. When this is the case, we have automatically~$(b, c)_F = 0$ from Lemma~\ref{lem-redundancies}. \hfill $\square$
\end{lem}

This concludes the proof of the Proposition. Given that the existence of~$\tilde B$ and~$\tilde C$ is itself controlled by an simple polynomial equation, as in Proposition~\ref{prop-D4-iff-cup-prop-iff-norm}, the situation is now entirely rewritten in terms of the existence of a rational point on an algebraic variety.

\subsection{Splitting varieties} \label{subsec-splitting-varieties}

We proceed to translate our results into the language of algebraic geometry.

As ever, let~$F$ be a field of characteristic not~$2$, let~$a, b, c, d \in F^\times$, and put~$\E= F[X, Y] / (X^2 - a, Y^2 - d)$. Together, Theorem~\ref{thm-main-four-cup-products},
Proposition~\ref{prop-D4-iff-cup-prop-iff-norm} and  Proposition~\ref{prop-equation-iff-four-cup-products}  show that the vanishing of the Massey product $\langle a, b, c, d \rangle$ is equivalent to the existence of a solution to the following system of equations, with unknowns~$x_1, y_1, z_1, x_2, y_2, y_3 \in F$, $u, v \in \E$ :
\begin{enumerate}
\item $x_1^2 - a y_1^2 = bz_1^2$,
\item $x_2^2 - dy_2^2 = cz_2^2$,
  \item $u^2\tilde B + v^2\tilde C = 1$, where~$\tilde B = x_1 + y_1 X \in \E$ and~$\tilde C= x_2 + y_2 Y \in \E$.
\end{enumerate}
There is also the condition~$z_1 \ne 0$, $z_2 \ne 0$. Here~$u = u_1 + u_2 X + u_3 Y + u_4 XY$, and likewise for~$v$, so that equation (3) can be written as four equations over~$F$ (carrying out the expansion in practice does not seem to clarify things).

For technical reasons, related to Proposition~\ref{prop-modifiers}, we change coordinates and work with the equations:
\begin{enumerate}
\item[(i)] $x_1^2 - a y_1^2 = b$,
\item[(ii)] $x_2^2 - dy_2^2 = c$,
  \item[(iii)] $u^2 \beta \tilde B + v^2 \gamma \tilde C = 1$, where~$\tilde B = x_1 + y_1 X \in \E$ and~$\tilde C= x_2 + y_2 Y \in \E$.
\end{enumerate}
Here~$\beta, \gamma \in F^\times$ are two new unknowns, and~$z_1, z_2$ have disappeared.

Equations (i), (ii), (iii) define an affine subvariety of~$\A^{14}_F$, the affine space of dimension 14 over~$F$; we consider its intersection with the open subset defined by~$\beta \ne 0$, $\gamma \ne 0$, and call it~$\mathscr X_F$. 
Also, note that equation (i) implies that $\tilde B$ is a \emph{unit} of $F_a := F[X]/(X^2-a)$, and similarly equation (ii) implies that $\tilde C$ is a \emph{unit} of $F_d := F[Y]/(Y^2-d)$.
In particular, we can describe $\mathscr X_F$ in a more conceptual way using the Weil-restriction functors $R_{\E|F}$, $R_{F_a|F}$ and $R_{F_d|F}$.
Namely, consider the $F$-variety:
\[
\mathscr{Y} := \Gm \times \Gm \times R_{\E/F} \A^2_{\E} \times R_{F_a/F}\Gm \times R_{F_d/F}\Gm \, . 
\]
We view an $F$-point of $\mathscr{Y}$ as a tuple $(\beta,\gamma,u,v,\tilde B,\tilde C)$, with $\beta,\gamma \in F^\times$, $u,v \in \E$, $\tilde B \in F_a^\times$ and $\tilde C \in F_d^\times$.
Then $\mathscr X_F$ is the closed subvariety of $\mathscr{Y}$ defined by the three equations:
\[ \Norm_{F_a/F}(\tilde B) = b, \ \Norm_{F_d/F}(\tilde C) = c, \ u^2 \beta \tilde B + v^2 \gamma \tilde C = 1 \, . \]

We have the following trivial, yet crucial property: for an extension $L/F$, we have
\[ \mathscr X_F \times_{\spec(F)} \spec(L) = \mathscr X_L \, .   \]
In particular, we have~$\mathscr X_F(L) = \mathscr X_L(L)$ (using the standard notation for the set of rational points). As already noted, the set~$\mathscr X_F(F)$ is non-empty if and only if the Massey product~$\langle a, b, c, d \rangle $ is defined and vanishes in the cohomology of~$F$. Now, it is obviously also true, but nicer, that for any field extension~$L/F$, the set~$\mathscr X_F(L)$ is non-empty if and only if the Massey product~$\langle a, b, c, d \rangle $ is defined and vanishes in the cohomology of~$L$. This is a property which is expected of a ``splitting variety'' for the problem of the vanishing of $\langle a, b, c, d \rangle$.

As it turns out, we can introduce a second splitting variety $X_F$. It will depend on choices, and so is not canonically associated with the problem alone; on the other hand, the local-global principle is established in the Appendix for~$X_F$ rather than~$\mathscr X_F$. The construction will echo Proposition~\ref{prop-modifiers} rather precisely. Let~$Z_F$ be the subvariety of $R_{F_a/F}\Gm \times R_{F_d/F}\Gm$, defined over~$F$ be the equations 
\[ \Norm_{F_a/F}(\tilde B) = b, \ \Norm_{F_d/F}(\tilde C) = c \, .   \]
There is an obvious morphism~$\pi\colon \mathscr X_F \to Z_F$ (forgetting~$\beta, \gamma, u, v$), and we will define~$X_F$ to be a fibre of~$\pi$ above an~$F$-rational point of~$Z_F$. That is, we suppose from now on that~$(a,b)_F = (c,d)_F = 0$, and using Proposition~\ref{prop-D4-iff-cup-prop-iff-norm} twice, we select~$\tilde B_0 \in F_a$ and~$\tilde C_0 \in F_d$ whose norms are~$b$ and~$c$ respectively; then we put~$X_F = \pi^{-1}(\tilde B_0, \tilde C_0)$.
The variety~$X_F$ is thus defined by the equation 
\[ \ u^2 \beta \tilde B_0 + v^2 \gamma \tilde C_0 = 1 \, .  \]
Note that our construction of $X_F$ depends on the choice of $\tilde B_0$ and $\tilde C_0$ as above.
In the sequel, this choice will always be clear from context, so we omit the $\tilde B_0$, $\tilde C_0$ from the notation.

It is clear that~$X_F$ is also compatible with base-change, just like~$\mathscr X_F$ is. That it is also a splitting variety is part of the next Theorem.

%% Note that at this stage, we care mostly about the equivalence of (1) and (3), but it is also convenient to have a summary of the situation, with the ``splitting variety'' point of view and the ``cup-product'' point of view side by side.

\begin{thm}
\label{thm-XF-splitting-variety}
Let the notation be as above (in particular, $\tilde B_0$ and~$\tilde C_0$ have been chosen). Let~$L/F$ be any field extension. The following statements are equivalent.
\begin{enumerate}
\item The Massey product~$\langle a, b, c, d \rangle $ is defined and vanishes in the cohomology of~$L$.
\item The variety~$\mathscr X_F$ has an~$L$-rational point.
  \item The variety~$X_F$ has an~$L$-rational point.
%% \item There exist $B \in L[\sqrt a]$ and~$C\in L[\sqrt d]$ with~$\Norm_{F[\sqrt a]/F}(B) = b$ and $\Norm_{F[\sqrt d]/F}(C) = c$, and there exist~$\beta, \gamma \in L$, such that we have $(\beta B, \gamma C)_{L[\sqrt a, \sqrt d]} = 0$, $(\beta B, c)_{L[\sqrt a]} = 0$, $(b, \gamma C)_{L[\sqrt d]} = 0$ and~$(b, c)_L = 0$ simultaneously.
  %%   \item Let~$B_0 \in F[\sqrt a]$ and~$C_0 \in F[\sqrt d]$ correspond to~$\tilde B_0$ and~$\tilde C_0$. There exist~$\beta, \gamma \in L$ such that we have $(\beta B_0, \gamma C_0)_{L[\sqrt a, \sqrt d]} = 0$, $(\beta B_0, c)_{L[\sqrt a]} = 0$, $(b, \gamma C_0)_{L[\sqrt d]} = 0$ and~$(b, c)_L = 0$ simultaneously.
\end{enumerate}

Moreover, let~$\beta, \gamma \in L$. Then there is a rational point~$(\beta, \gamma, u, v) \in X_F(L)$ if and only if we have $(\beta B_0, \gamma C_0)_{L[\sqrt a, \sqrt d]} = 0$, $(\beta B_0, c)_{L[\sqrt a]} = 0$, $(b, \gamma C_0)_{L[\sqrt d]} = 0$ and~$(b, c)_L = 0$ simultaneously, where~$B_0 \in L[\sqrt a]$ corresponds to~$\tilde B_0$ under the natural map~$F_a \to L[\sqrt a]$, and likewise for~$C_0 \in L[\sqrt d]$.
\end{thm}

\begin{proof}
  Since~$\mathscr X_F(L) = \mathscr X_L(L)$, and similarly for~$X_F$, we may as well (and we do) assume that~$L=F$. Then the equivalence  is a mere reformulation of earlier material.  To wit, Theorem~\ref{thm-main-four-cup-products} together with Proposition~\ref{prop-equation-iff-four-cup-products} shows the equivalence of (1) and (2), the variety~$\mathscr X_F$ being defined just for this purpose. The implication (3) $\implies$ (2) is trivial. On the other hand, Proposition~\ref{prop-modifiers} gives (2) $\implies$ (3) readily.

  The ``moreover'' statement is obtained by another application of Proposition~\ref{prop-equation-iff-four-cup-products}. 
\end{proof}

Having such a statement dealing with field extensions is necessary for us, as we intend to apply a local-global principle to prove the existence of rational points, and this requires an understanding of~$X_F(F_v)$ where~$F_v$ is a completion of~$F$.

\section{The $4$-Massey Vanishing Conjecture for number fields} \label{sec-applications}

In this section we finally prove:

\begin{thm} \label{thm-main-number-fields}
Let~$F$ be a number field, and let~$a, b, c, d \in F^\times$ be such that the Massey product~$\langle a, b, c, d\rangle$ is defined. Then $\langle a, b, c, d\rangle $ vanishes. In other words, the $4$-Massey Vanishing Conjecture is true for number fields.
\end{thm}

\begin{proof}
  Since the Massey product is defined, we can apply Theorem~\ref{thm-U5-implies-rational-point}, but some simplifying remarks are in order. First, we have $(a,b)=(b,c)=(c,d)=0$, and so Lemma~\ref{lem-when-abcd-square} takes care of the case when one of~$a, b, c, d$ is a square in~$F$; now we assume that none of them is a square, and in particular, we identify~$F_a$ and~$F[\sqrt a]$, and we identify~$F_d$ and~$F[\sqrt d]$. We do not distinguish between~$\tilde B$ and~$B$, or between~$\tilde C$ and~$C$, in the notation of Theorem~\ref{thm-U5-implies-rational-point}. A second point is that we may replace once and for all~$b$ by~$bz_1^2$ for~$z_1 \in F^\times$ if we wish, as the class~$\chi_b \in \h^1(F, \f_2)$ is not affected by this, and neither is the Massey product~$\langle a, b, c, d \rangle$. Likewise with~$c$.

With these precautions, the result of our application of Theorem~\ref{thm-U5-implies-rational-point} is this. We can find~$B_0 \in F[\sqrt a]$ and~$C_0 \in F[\sqrt d]$ such that~$\Norm_{F\sqrt a]/F}(B_0) = b$ and~$\Norm_{ F[\sqrt d]/F }(C_0) = c$, while~$(B_0, c)_{F[\sqrt a]} = 0$ and~$(b, C_0)_{F[\sqrt d]} = 0$. Finally, we can find~$u \in \h^2(F, \f_2)$ whose restriction to~$E$ is~$(B_0, C_0)$.

We select this $B_0$ and this~$C_0$ in order to construct the splitting variety~$X_F$, and we proceed to prove that~$X_F(F)$ is non-empty (by Theorem~\ref{thm-XF-splitting-variety}, we will then be done). From Theorem~\ref{app:th:numberfield}, we see that it suffices to show that for each place~$v$ of~$F$, we can find a rational point $(\beta_v, \gamma_v, u_v, v_v )$ in~$X_F(F_v)$, in such a way that our various choices satisfy 
\[ \sum_v \inv_v\left( (\beta_v, c)_{F_v} \right) = 0 \, , \qquad \sum_v \inv_v \left( (b, \gamma_v)_{F_v} \right) = 0 \, .   \]
Here~$\inv_v$ is the unique isomorphism between~$\h^2(F_v, \f_2)$ and~$\f_2$. We turn to this, and in fact we shall arrange to have~$\inv_v\left( (\beta_v, c)_{F_v} \right) = 0$ and $\inv_v \left( (b, \gamma_v)_{F_v} \right) = 0$ at each place.

Let~$v$ be a place. We see~$B_0$ and~$C_0$, chosen above, as elements of~$F_v[\sqrt a]$ and~$F_v[\sqrt d]$ respectively, and we wish rely on the ``moreover'' statement of Theorem~\ref{thm-XF-splitting-variety} with~$L= F_v$. 

First we treat the case when one of~$a$ or~$d$ is not a square in the completion~$F_v$: either way, the field~$F_v[\sqrt a, \sqrt d]$ is strictly larger than~$F_v$. However, it is a well-known fact from the theory of local fields that the restriction map~$\h^2(F_v, \f_2) \to \h^2(F_v[\sqrt a, \sqrt d], \f_2)$ is then the zero map. Since~$(B_0, C_0)_{F_v[\sqrt a, \sqrt d]}$ is the image of~$u_{F_v} \in \h^2(F_v, \f_2)$ (in the above notation), we have in fact $(B_0, C_0)_{F_v[\sqrt a, \sqrt d]} = 0$. For such a place, we take~$\beta_v = \gamma_v = 1$. The four cup-products in (3) of Theorem~\ref{thm-XF-splitting-variety} then vanish, so we have a rational point in~$X_F(F_v)$. In this case~$(\beta_v, c)_{F_v} = 0$ and~$(b,\gamma_v)_{F_v} = 0$, as promised.

  Now suppose alternatively that~$a$ and~$d$ are both squares in~$F_v$. Suppose our element~$B_0$ was of the form~$B_0 = x + y \sqrt a \in F[\sqrt a] \subset F_v$, and let~$\beta_v = x - y \sqrt a \in F_v$, so that~$\beta_v B_0 = b$. Likewise, write~$C_0 = x' + y' \sqrt d$ and let~$\gamma_v = x' - y' \sqrt d \in F_v$, so~$\gamma_v C_0 = c$. The four cup-products mentioned in (3) of Theorem~\ref{thm-XF-splitting-variety} are equal to~$(b,c)_{F_v}$, so they all vanish, and there is an~$F_v$-rational point. Moreover, we have 
\[ 0 = (\beta_v B_0, c)_{F_v} = (\beta_v, c)_{F_v} + (B_0 , c)_{F_v} = (\beta _v, c)_{F_v} \, ,   \]
as~$(B_0, c)_{F[\sqrt a]} = 0$. Similarly we draw~$(b, \gamma_v)_{F_v} = 0$. 
\end{proof}

The argument given in this proof establishes, in particular, that fourfold Massey products always vanish in the cohomology of local fields, when they are defined. This was of course known (we quote a strong version of this in the next proof), but here everything stays fairly concrete.

In many cases, we obtain an improved version of the conjecture:

\begin{thm} \label{thm-main-number-fields-generic}
  Let $F$ be a number field, and let $a,b,c,d \in F^\times$ be given. Suppose that none of~$ad$, $ab$, $cd$ is a square. Then the following are equivalent:
  \begin{enumerate}
    \item The $4$-Massey product $\langle a,b,c,d \rangle$ vanishes.
    \item The $4$-Massey product $\langle a,b,c,d \rangle$ is defined.
    \item One has $(a,b)_F = (b,c)_F = (c,d)_F = 0$.
  \end{enumerate}
\end{thm}

\begin{proof}
  The definitions are so arranged that $(1) \implies (2)$ is a tautology, while $(2) \implies (3)$  is (a very small) part of Theorem~\ref{thm-U5-implies-rational-point} (and following remark). The non-trivial portion of the proof is $(3) \implies (1)$.

Assume (3).  From Theorem~\ref{thm-XF-splitting-variety}, we must prove that~$X_F(F)$ is non-empty. By Theorem~\ref{app:th:numberfield} in the Appendix, it suffices to show that for each place~$v$ of~$F$, we have~$X_F(F_v) \ne \emptyset$. Applying Theorem~\ref{thm-XF-splitting-variety} yet again, we now see that we must prove that the Massey product~$\langle a, b, c, d \rangle$ is defined and vanishes in the cohomology of~$F_v$. However, for a local field such as~$F_v$, it is known (see~\cite{MTlocal}, Proposition 4.1) that the conditions~$(a,b)_{F_v} = (b,c)_{F_v} = (c,d)_{F_v} = 0$ (which hold here by restriction of the analogous identities over~$F$) are enough to imply this.
\end{proof}

\begin{rmk}
  \label{rmk-generic}
  The implication $(3) \implies (1)$ from Theorem \ref{thm-main-number-fields-generic} fails in general, if one removes the additional assumptions on $ad$, $ab$, $cd$.
  More precisely, Proposition \ref{app:prop:unramifiedrs} and Remark \ref{app:rmk:unram-iff} provide necessary and sufficient conditions for the classes $(\beta,c)$, $(\gamma,b)$ and/or $(\beta,c)+(\gamma,b)$, over the function field of $X_F$, to be unramified over $F$ (see the notation in \S\ref{subsec-splitting-varieties}).
  When these conditions hold for one of these classes, a Brauer-Manin obstruction to the implication $(3) \implies (1)$ may arise.
  See Example \ref{app:example}, suggested by Y. Harpaz, for a concrete situation where (3) holds, so the variety $X_F$ has local points everywhere, but no $F$-rational point exists -- that is, (1) fails. By Theorem~\ref{thm-main-number-fields}, (2) also fails to hold. (In this case $(\beta,c)$ is unramified.)
  As we see from Theorem~\ref{thm-main-number-fields-generic}, this cannot happen in {\em non-degenerate} situations, where $[F(\sqrt{a},\sqrt{b},\sqrt{c},\sqrt{d}):F] = 2^4$.
\end{rmk}

\section{Explicit constructions} \label{sec-explicit}

Suppose~$a, b, c, d \in F^\times$ are linearly independent in~$\li F$. When the Massey product $\langle a, b, c, d \rangle$ vanishes, there is a certain map~$\phi \colon G_F \to \uf$ which is surjective when composed with~$\uf \to \uf / \Phi(\uf)$, the quotient modulo the Frattini subgroup, as follows from examining the notation (note that~$\Phi (\uf)$ is the intersection of the kernels of the four maps~$s_i$, $i= 1, 2, 3, 4$). If follows that~$\phi $ is itself surjective, and therefore, there exists an extension~$L/F$ such that~$\gal(L/F) \cong \uf$. The compatibility with~$a, b, c, d$ means that~$F[\sqrt a, \sqrt b, \sqrt c, \sqrt d]$ must be contained in~$L$, corresponding to the Frattini quotient via the Galois correspondence.

This section is about constructing~$L$ explicitly, under the condition that our usual four cup-products vanish -- or equivalently, from Lemma~\ref{lem-redundancies}, under the condition~$(B, C)_E=0$. As it turns out, we end up giving an alternative, more explicit proof for Theorem~\ref{thm-rational-point-implies-U5}, restricted to the ``non-degenerate case''.

\begin{thm} \label{thm-conditions-sufficient-explicit}
Let~$F$ be a field of characteristic~$\ne 2$, and let~$a, b, c, d \in F^\times$ be elements such that~$\cl a, \cl b, \cl c, \cl d$ are linearly independent in~$\li F$.

Assume that we can find~$x, y \in F$ such that
\[ x^2 -a y^2 = b \, ,  \tag{1}  \]
and likewise assume that we can find~$x', y' \in F$ such that
\[ (x')^2 - d (y')^2 = c \, . \tag{2}  \]

Finally, put~$B= x + y \sqrt a$ and~$C = x' + y' \sqrt d$, and assume that we can find~$u, v \in F[\sqrt a, \sqrt d]$ such that
\[ u^2 - Bv^2 = C \, . \tag{3}  \]
Under these assumptions, if we put~$w = u + v \sqrt{B}$, then the Galois closure~$L$ of
\[ F[\sqrt a, \sqrt b, \sqrt c, \sqrt d, \sqrt B, \sqrt C, \sqrt w]  \]
verifies~$\gal(L/F) \cong \uf$.
\end{thm}

The rest of this section is devoted to the proof. The argument is self-contained, but assumes the notation from \S\ref{sec-defs} and \S\ref{sec-cup-products}, and uses Shapiro's lemma.

\begin{lem}
Put~$K = F[\sqrt a, \sqrt b, \sqrt c, \sqrt d, \sqrt B, \sqrt C]$. Then~$K/F$ is Galois with $$\gal(K/F) \cong D_4 \times D_4 \, . $$

More precisely, the isomorphism can be chosen such that the standard generating involutions~$\sigma_1, \sigma_2, \sigma_3, \sigma_4 \in D_4 \times D_4$, when viewed in~$\gal(K/F)$, act on the elements of~$K$ as follows:
\[ \begin{array}{c|r|r|r}
    & \sqrt a     & \sqrt b   & \sqrt B \\
\hline
\sigma_1   & - \sqrt a & \sqrt b & \star \\
\sigma_2   & \sqrt a &  - \sqrt b & - \sqrt B \\{}
[\sigma_1, \sigma_2]   & \sqrt a &  \sqrt b & - \sqrt B
\end{array} \quad\textnormal{and}\quad \begin{array}{c|r|r|r}
    & \sqrt d     & \sqrt c   & \sqrt C \\
\hline
\sigma_4   & - \sqrt d & \sqrt c & \star \\
\sigma_3   & \sqrt d &  - \sqrt c & - \sqrt C \\{}
[\sigma_4, \sigma_3 ]  & \sqrt d &   \sqrt c & - \sqrt C
\end{array}\]
Finally, $\sigma_1$ and $\sigma_2$ fix~$\sqrt c, \sqrt d, \sqrt C$, and vice-versa.
\end{lem}

The $\star$ means that we do not insist on a value. One may show, for example, that $\sigma_1(\sqrt B) = \pm \frac{\sqrt b} {\sqrt B}$, but the particular sign will not be relevant.

\begin{proof}
Let~$K_1 = F[\sqrt a, \sqrt b, \sqrt B]$ and~$K_2 = F[\sqrt c, \sqrt d, \sqrt C]$. Then~$K_i$ is a~$D_4$-extension of~$F$, for~$i= 1, 2$, and the actions are as announced, for some choices of generating involutions for the dihedral groups : simply argue as in the proof of Proposition~\ref{prop-D4-iff-cup-prop-iff-norm}.

In order to show that~$K= K_1K_2$ is a~$D_4 \times D_4$-extension, it suffices to show that~$K_1 \cap K_2 = F$.

The extension of~$F$ which corresponds to the Frattini quotient of~$\gal(K_1/F)$ resp.\ $\gal(K_2/F)$ is~$F[\sqrt a, \sqrt b]$, resp.\ $F[\sqrt c, \sqrt d]$, so~$K_1 \ne K_2$. Thus~$K_1 \cap K_2$ corresponds to a non-trivial normal subgroup of~$\gal(K_i/F)$, for~$i= 1, 2$. Looking at the normal subgroups of~$D_4$, we see that~$K_1\cap K_2 \subset F[\sqrt a, \sqrt b] \cap F[\sqrt c, \sqrt d] = F$, which concludes the proof.
\end{proof}

From now on we write~$G$ for the group~$D_4 \times D_4$, which we have just identified explicitly with~$\gal(K/F)$. As above, we will write~$N$ for the subgroup of~$\gal(K/F)$ generated by~$\sigma_2, [\sigma_1, \sigma_2], \sigma_3, [\sigma_4, \sigma_3]$.

\begin{lem}
The fixed field of~$N$ within~$K$ is~$E = F[\sqrt a, \sqrt d]$. \hfill $\square$
\end{lem}

Consider now~$w = u + v \sqrt B$ as above, with~$u, v \in F[\sqrt a, \sqrt d]$, and let us study its class~$\cl w \in \li K$.
Note that $w$ is clearly non-zero, since~$w(u - v\sqrt B) = C \ne 0$.

\begin{lem}
The element~$\cl w$ is fixed by~$N$.
\end{lem}

\begin{proof}
The element~$w$ is itself fixed by~$\sigma_3$ and~$[\sigma_4, \sigma_3]$, as we see immediately from the tables. On the other hand $\sigma_2(w) = u - v \sqrt{B}$, and things have thus been arranged so that
\[ \sigma_2(w) = \frac{w \sigma_2(w)} {w} = \frac{C} {w} = w \left(\frac{\sqrt C} {w}\right)^2 \, .   \]
As a result $\sigma_2(\cl w) = \cl w$. The element~$[\sigma_1, \sigma_2]$ has the same effect on~$w$ as~$\sigma_2$, so the same argument applies.
\end{proof}

We now let~$W$ denote the~$G$-module spanned by~$\cl w$ within~$\li K$. By the Lemma this can be seen as a~$G/N$-module, and indeed it is the image of
\[ \pi \colon \f_2[G/N] \longrightarrow W  \]
mapping~$1 \in G/N$ to~$\cl w$. We put~$L = K[\sqrt W]$, which is indeed the~$L$ introduced in the Theorem. Equivariant Kummer theory states that~$L/F$ is Galois, and that~$\gal(L/K) \cong W^*$ {\em via} the Kummer pairing. For future reference, let us recall that~$\tau \in \gal(L/K)$ is viewed as an element of~$W^*$ via
\[ \tau (w) = \frac{\tau(\sqrt w)} {\sqrt w} \in \{ \pm 1 \} \cong \f_2 \, .   \]

So we have an exact sequence
\[ 0 \longrightarrow W^* \longrightarrow \gal(L/F) \xrightarrow[\quad]{q} G \longrightarrow 1 \, , \]
and we let~$\alpha \in \h^2(G, W^*)$ denote the corresponding cohomology class. Consider the following commutative diagram.
\[ \begin{CD}
  \h^2(G, W^*) @>{\pi^*}>> \h^2(G, \f_2[G/N]^*) \\
  @V{\widetilde{sh}}VV              @VV{sh}V \\
  \h^2(N, \frac{W^*} {[w]^\perp})  @>=>>   \h^2(N, \f_2) \, .
\end{CD}
  \]
  Here is a word of explanation about the notation. First, $[w]^\perp = \{ f \in W^* : f([w])= 0 \}$. The map~$sh$ is Shapiro's isomorphism; the map~$\pi^*$ is the injection which is dual to the surjection~$\pi \colon \f_2[G/N] \to W$; the map~$\widetilde{sh}$ is Shapiro-like, defined by using the projection~$W^* \to \frac{W^*} {[w]^\perp}$ and restriction to~$N$; and the bottom map is seen as the identity, since~$\frac{W^*} {[w]^\perp}$ can be identified with~$\f_2$ in a unique way (just like any group of order~$2$). Commutativity is clear.

The cohomology class~$\widetilde{sh}(\alpha )$ corresponds to the extension
 \[ 0 \longrightarrow \frac{W^*} {[w]^\perp} \longrightarrow \frac{q^{-1}(N)} {[w]^\perp } \longrightarrow N \longrightarrow 1 \, . \tag{$\dagger$}  \]
Let us elucidate a few things. First we have
\[ q^{-1}(N) = \gal(L/E) \, .   \]
Further, \begin{align*}
[w]^\perp & = \{ f \in W^* : f([w]) = 0 \} \\
       & \cong \{ \tau  \in \gal(L/K) : \tau  (\sqrt w) = \sqrt w \} \, .
\end{align*}
It follows that, if we put~$L' = K[\sqrt w]$, then~$\gal(L/L') \cong [w]^\perp$, and
\[ \frac{q^{-1}(N)} {[w]^\perp} = \gal(L'/E) \, .   \]

\begin{lem} \label{lem-gal-L-prime-E}
The Galois group~$\gal(L'/E)$ is isomorphic to~$C_2^2 \times D_4$.
\end{lem}

\begin{proof}
The element~$w \in E[\sqrt B]$ satisfies~$N_{E[\sqrt B]/E}(w) = C$, and it follows that
\[ M=E[\sqrt B, \sqrt C, \sqrt w] \]
is a~$D_4$-extension of~$E$. (Here we do use the fact that~$E[\sqrt B, \sqrt C]$ is a~$C_2^2$-extension of~$E$, which we know from~$\gal(K/E)= N$.)

Of course~$K/E$ is Galois with group~$N \cong C_2^4$, so that~$L' = KM$ is Galois with group
\[ \gal(KM/E)= \gal(K/E) \times_{\gal(K\cap M /E)} \gal(M/E) \, .   \]

Certainly we have~$E[\sqrt B, \sqrt C] \subset K \cap M$, and the intersection $K\cap M$ cannot in fact be larger, for (by counting dimensions, say) if it were we would have $K\cap M = M$; so $M \subset K$, a non-abelian extension contained in an abelian one, contradiction. So $K\cap M = E[\sqrt B, \sqrt C]$.

Thus~$\gal(L'/E)$ is of the form~$C_2^4 \times_{C_2^2} D_4$. Since the map~$C_2^4 \to C_2^2$ is split, we see that this fibre product is in fact isomorphic to~$C_2^2 \times D_4$.
\end{proof}

From this last Lemma, and its proof, it follows that ($\dagger$) can be identified with the first exact sequence in Lemma~\ref{lem-classes-extensions-of-N}, which implies that~$\widetilde{sh}(\alpha )= x_2 x_3$. As a consequence, from the commutativity of the diagram above and the injectivity of Shapiro's map, we see that $\pi^*(\alpha )$ describes the extension
\[ 0 \longrightarrow S\cong \f_2[G/N]^* \longrightarrow \uf \longrightarrow G \longrightarrow 1 \, .   \]
We conclude that~$\gal(L/F)$ is a subgroup of~$\uf$ which maps onto the Frattini quotient~$\gal(F[\sqrt a, \sqrt b, \sqrt c, \sqrt d]/F)$. Thus~$\gal(L/F)= \uf$.

\begin{ex}
Let us take~$F = \q$ and $a= 11$, $b=5$, $c=79$, $d=13$. Let us try to look for solutions to (1)-(2)-(3). A little calculation with Hilbert symbols reveals that~$(a,b)=(b,c)=(c,d)=0$, and so Theorem~\ref{thm-main-number-fields} guarantees that such solutions do exist.

In practice though, the quickest way to look for~$x, y, z \in \q$ such that
\[ x^2 -11 y^2 = 5 z^2   \]
is to pick random integers~$y$ and~$z$ until~$11 y^2 + 5 z^2$ is a square. Let us fix an initial solution, say~$x_0 = 4$, $y_0 = z_0 = 1$, and put~$B_0= 4 + \sqrt{11}$. Likewise~$C_0 = 14 + 3 \sqrt{13}$ is an initial solution to equation (2).

With these random values, there will likely be no solution to (3), since~$(B_0, C_0) \ne 0$ in general. The next step is to pick random integers~$f, g \in \z$ and check whether~$(fB_0, gC_0) = 0$. Again, Proposition~\ref{prop-modifiers} justifies the existence of these -- but there would be no harm in trying even if we did not know that. The most efficient method seems to be to compute the conductor of~$(fB_0, gC_0)$, that is, the product of those prime ideals in the ring of integers of~$E = \q[\sqrt {11}, \sqrt{13}]$ such that the cup product~$(fB_0, gC_0)$ maps to a non-zero class in the corresponding non-archimedean completion: this is an operation done quickly by the software PARI (which we have used through Sagemath). If the conductor is trivial, we use the more time-consuming methods of PARI to find a rational point on the conic defined by (3).

Having found~$f, g$, we put~$B= fB_0$ and~$C= gC_0$ (but also in principle~$x=fx_0, y= fy_0, z= fz_0$, and so on, thus changing our solutions to (1), (2), although these do not show up themselves in the statement of Theorem~\ref{thm-conditions-sufficient-explicit}).

In the case at hand, one search of the type just described has yielded the solutions
\[ B=\frac{1}{7} \, \sqrt{11} + \frac{16}{7}\, , \qquad C=\frac{9}{2} \, \sqrt{13} + \frac{37}{2} \]
which verify~$(B, C)= 0$. Indeed, equation (3) is solved by~$w = u + v \sqrt{B}$ with
\[ u=\frac{2}{57319} \, \sqrt{13} {\left(40730430348570235670 \, \sqrt{11} - 283850079471799786881\right)} \]
\[ -\frac{642122498218267058484}{57319} \, \sqrt{11} + \frac{143760709913945809313}{7396} \, ,  \]
\[ v=\frac{4}{286595} \, \sqrt{13} {\left(27413052840094197823 \, \sqrt{11} - 191041370339652873536\right)} \]
\[ -\frac{172868666747038399008}{57319} \, \sqrt{11} + \frac{193512315164122974131}{36980} \, . \]

\end{ex}

\appendix

\section{Local-global principles and the splitting varieties}

\begin{center}
{\sc by Olivier Wittenberg\footnote{D\'epartement de math\'ematiques et applications, \'Ecole normale sup\'erieure, 45~rue d'Ulm, 75230 Paris Cedex 05, France. {\tt wittenberg@dma.ens.fr}} }
\end{center}

%\section{Local-global principles and the splitting varieties\texorpdfstring{\\\medskip}{ (}by Olivier Wittenberg\texorpdfstring{}{)}}

The goal of this Appendix is to establish a local-global principle,
when~$F$ is a number field,
for the existence of a rational point on the variety~$X_F$ appearing in
Theorem~\ref{thm-XF-splitting-variety} 
of the paper (see Theorem~\ref{app:th:numberfield} below).
I am indebted to the authors for sharing drafts of their paper with me,
to Pierre Guillot
for numerous discussions on $4$\nobreakdash-Massey products,
and to Yonatan Harpaz for suggesting Example~\ref{app:example}.

\subsection{Statements}
\label{app:sub:statements}

Let~$F$ be a field of characteristic zero.
For $q,q' \in F$, we set $F_q=F[t]/(t^2-q)$
and $F_{q,q'} = F[t,t']/(t^2-q,t'^2-q')$.
Let us fix $a,b,c,d \in F^*$
and $B\in F_a^*$, $C \in F_d^*$
such that $N_{F_a/F}(B)=b$ and $N_{F_d/F}(C)=c$,
and consider the 
closed subvariety
\begin{align*}
X \subset \Gm^2 \times R_{F_{a,d}/F} \A^2_{F_{a,d}}
\end{align*}
defined by the equation $\beta B x^2 + \gamma C y^2 = 1$,
where $R_{F_{a,d}/F}$ denotes the Weil restriction functor
and $\beta,\gamma$ are the coordinates of $\Gm^2$ while $x,y$ are those of $R_{F_{a,d}/F}\A^2_{F_{a,d}}$. (In the body of the paper, this variety is denoted by~$X_F$ rather than~$X$, and our~$B, C$ play the r\^ole of the elements called~$B_0, C_0$ there.)

When~$F$ is a number field, we denote by~$\Omega$ the set of its places,
by~$F_v$ the completion of~$F$ at $v \in \Omega$,
and by $(x,y) \in \{-1,1\}$
the Hilbert symbol of $x,y \in F_v^*/F_v^{*2}$ (see \cite[Ch.~V, \textsection3]{neukirchant}).

\begin{thm}
\label{app:th:numberfield}
Assume that~$F$ is a number field.
If none of $ad$, $ab$, $cd$ is a square in~$F$, then~$X$ satisfies the Hasse principle:
$X(F)\neq\emptyset$ if and only if $X(F_v)\neq\emptyset$ for all $v \in \Omega$.
In any case, the existence of a rational point on~$X$
is equivalent to the existence
of an element of $\prod_{v \in\Omega} X(F_v)$
whose $\beta$ and $\gamma$ coordinates $\beta_v,\gamma_v \in F_v^*$
satisfy
\begin{align*}
\prod_{v \in \Omega} (\beta_v,c)=\prod_{v \in \Omega} (\gamma_v,b)=1\rlap{.}
\end{align*}
In addition, if~$X(F)\neq\emptyset$, then $X(F)$ is dense in~$X$ for the Zariski topology.
\end{thm}

It is not hard to see that~$X$ is smooth, irreducible, and geometrically rational.
When~$F$ is a number field, the existence of a rational point on~$X$
is therefore conjectured to be controlled by the Brauer--Manin obstruction
(see~\cite[\textsection4]{cttoulouse}).
To establish Theorem~\ref{app:th:numberfield},
we shall first deduce
the validity of this conjecture,
in the case of~$X$,
by an application of the fibration method
(specifically, of \cite[Th.~9.31]{hw}).
We shall then prove, in Theorem~\ref{app:th:anyfield} below, 
that the unramified Brauer group of~$X$ consists of constant classes,
except when $ad$, $ab$, or~$cd$ is a square, in which case the classes of
the quaternion algebras $(\beta,c)$ and $(\gamma,b)$ over~$F(X)$ may come into play.
The combination of these two facts yields Theorem~\ref{app:th:numberfield}.
We note that Theorem~\ref{app:th:anyfield} is a purely algebraic statement:
it holds over an arbitrary field~$F$ of characteristic zero.

\begin{thm}
\label{app:th:anyfield}
The natural map $\Br(F) \to \Br_\nr(F(X)/F)$ is surjective
if none of $ad$, $ab$, $cd$ is a square in~$F$.
In any case, its cokernel is killed by~$2$ and is contained in the subgroup of $\mathrm{Coker}\mkern1mu\big(\Br(F)\to\Br(F(X))\big)$
generated by the classes of the quaternion algebras~$(\beta,c)$ and~$(\gamma,b)$ over~$F(X)$.
\end{thm}

We recall that the unramified Brauer group $\Br_\nr(K/F)$ of a finitely generated
field extension~$K/F$ is the intersection
of the subgroups $\mathrm{Im}\big(\Br(A) \to \Br(K)\big)$ where~$A$ ranges over the discrete
valuation rings $F \subset A \subset K$ with quotient field~$K$.
If~$F$ has characteristic zero and
$K=F(S)$ for a smooth irreducible variety~$S$,
the group $\Br(S)=\h^2_\et(S,\Gm)$
can be identified with the intersection of the subgroups $\mathrm{Im}\big(\Br(\sO_{S,\xi}) \to \Br(K)\big)$
where~$\xi$ ranges over the codimension~$1$ points of~$S$;
if~$S$ is proper, then $\Br(S)=\Br_\nr(K/F)$.
See
\cite[III, \textsection6]{grbr}, \cite[\textsection5]{ctsanmumbai}.

\begin{rmk}
It is only for simplicity that we assume that~$F$ has characteristic zero
in the statement of Theorem~\ref{app:th:anyfield}:
it allows us to refer to
a smooth compactification of~$X$.
When~$F$ has characteristic $p>2$,
the definition of~$X$ still makes sense and Theorem~\ref{app:th:anyfield} remains true.
Indeed, on the one hand,
the proof given below
easily adapts
to show that the 
cokernel of $\Br(F) \to \Br_\nr(F(X)/F)$
satisfies the desired statement modulo its $p$\nobreakdash-primary torsion subgroup,
and on the other hand,
this cokernel is killed by a power of~$2$
since~$X$ becomes rational over
$F(\sqrt{a},\sqrt{b},\sqrt{c},\sqrt{d})$
(see \cite[Prop.~1.7]{saltmanbrauer}).
Presumably, the proof of Theorem~\ref{app:th:numberfield} should
also work over a global field of characteristic $p>2$; however, the
results we use from~\cite{hw} have not been written down in this setting.
\end{rmk}

\subsection{Geometry}
\label{app:subsec:geometry}

Following~\cite{skorodescent}, we shall say that a scheme of finite type over a field is \emph{split}
if it possesses an irreducible component of multiplicity~$1$ which is
 geometrically irreducible.
In this preliminary section, we compactify~$X$ to the total space of a fibration,
over a proper base,
with very few non-split fibres in codimension~$1$.
This fibration will play a crucial r\^ole both in the proof of Theorem~\ref{app:th:numberfield}
and in that of Theorem~\ref{app:th:anyfield}.
We then proceed to make further observations concerning its fibres
(Proposition~\ref{app:prop:nulifts} below),
for use in the proof of Theorem~\ref{app:th:anyfield}.

Let $X'' \subset \Gm^2 \times R_{F_{a,d}/F} \P^2_{F_{a,d}}$
denote the closed subvariety
defined by $\beta B x^2 + \gamma C y^2 = z^2$,
where $\beta,\gamma$ are the coordinates of~$\Gm^2$
and $x,y,z$ now denote the homogeneous coordinates of $R_{F_{a,d}/F}\P^2_{F_{a,d}}$.
Clearly~$X''$ is smooth and contains~$X$ as a dense open subset.
Let us fix once and for all a smooth compactification
$X'' \subset X'$ such that
the map
$\phi:X'' \to \Gm^2$
defined by
$\phi(\beta,\gamma,[x:y:z])=(\beta,\gamma)$
extends to a morphism $\phi':X'\to \P^1_F \times \P^1_F$.
(We view $\Gm^2$ as an open subset of $\P^1_F \times \P^1_F$.)
For $w \in F^*$,
let $\nu_{\gamma=w},\nu_{\beta=w\gamma} \in \Gm^2$ denote the
generic points of the subvarieties of~$\Gm^2$ defined by $\gamma=w$ and by $\beta=w\gamma$,
respectively. The next proposition lists sufficient conditions for the fibre of~$\phi$
(or, equivalently, of~$\phi'$) above
these points to contain a rational point (\emph{i.e.}, an $F(\Gm)$\nobreakdash-point)
for some~$w$.

\begin{prop}
\label{app:prop:nulifts}
The following statements hold:
\begin{enumerate}
\item If~$c$ is a square in~$F$ or if~$d$ and~$ac$ are squares in~$F$,
\label{app:item:c}
then there exists $w \in F^*$ such that $\phi^{-1}(\nu_{\gamma=w})$,
 $\phi^{-1}(\nu_{\gamma=aw})$, and
 $\phi^{-1}(\nu_{\gamma=dw})$
contain a rational point.
\item
If $a=b=c=d$ in $F^*/F^{*2}$, 
then there exists $w\in F^*$ such that
$\phi^{-1}(\nu_{\beta=w\gamma})$
and $\phi^{-1}(\nu_{\beta=aw\gamma})$
contain a rational point.
\item In the case where~$d$ and~$ac$ are squares in~$F$,
one can take $w=C_1$ (or $w=C_2$) in~\eqref{app:item:c},
where $C=(C_1,C_2)$ denotes the image of~$C$ by the isomorphism
$F_d=F\times F$
induced by the choice of a square root of~$d$.
\end{enumerate}
\end{prop}

\begin{proof}
We first assume that~$c$ is a square in~$F$ and prove~(1) in this case.
As $N_{F_d/F}(C)=c$ is a square in~$F$,
it follows from Hilbert's Theorem~90
that there exists $w \in F^*$ such that $wC$ is a square in~$F_d$.
Evaluating the function~$\gamma C$ at any
$\nu \in \{\nu_{\gamma=w},\nu_{\gamma=aw},\nu_{\gamma=dw}\}$
yields an element of $F(\nu) \otimes_F F_d$ which becomes a square in $F(\nu)\otimes_F F_{a,d}$,
hence the claim.

Let us now assume that~$d$ and~$ac$ are squares in~$F$,
and prove~(1) and~(3) simultaneously.
In this case, following the notation of~(3),
we have $C_1C \in F_{a,d}^{*2}$ as~$c$ is a square
in~$F_a$
and as the decomposition $F_{a,d}=F_a \times F_a$
maps~$C_1C$ to~$(C_1^2,c)$.  Hence, with $w=C_1$, the value of~$\gamma C$ at
any $\nu \in \{\nu_{\gamma=w},\nu_{\gamma=aw},\nu_{\gamma=dw}\}$
is again a square in $F(\nu)\otimes_F F_{a,d}$.

Finally, let us assume that $a=b=c=d$ in $F^*/F^{*2}$ and turn to~(2).
The choice of a square root of~$ad$
determines isomorphisms $\iota:F_d \isoto F_a$ and $F_{a,d}=F_a \times F_a$.
As $N_{F_a/F}(B\iota(C))=bc$ is a square in~$F$, Hilbert's Theorem~90 ensures
the existence of $w \in F^*$ such that $-wB\iota(C)$ is a square in~$F_a$.
As~$c$ is a square in~$F_a$, it follows that~$-wBC$ is a square in~$F_{a,d}$
and hence that the value of the function $-\gamma C \beta B$
at any $\nu \in \{\nu_{\beta=w\gamma},\nu_{\beta=aw\gamma}\}$
is a square in $F(\nu) \otimes_F F_{a,d}$.
Hence, writing $(\beta B, \gamma C)$ for the class in
 $\Br(F(\nu) \otimes_F F_{a,d})$ of the corresponding quaternion
algebra, we have
$(\beta B,\gamma C)=(-\gamma C \beta B,\gamma C)=0$
for any such~$\nu$,
or equivalently $\phi^{-1}(\nu)$ possesses a rational point.
\end{proof}

\subsection{Arithmetic}

Let us deduce Theorem~\ref{app:th:numberfield}  from Theorem~\ref{app:th:anyfield}.
The relevant arithmetic input is the following statement.

\begin{thm}
\label{app:th:arithmeticinput}
Let~$Z$ be a smooth, proper, and irreducible variety over a number field~$F$.
Let $n \geq 1$ be an integer and $f:Z\to (\P^1_F)^n$ be a dominant morphism.
Assume that the geometric generic fibre of~$f$ is
rationally connected
and that the fibre of~$f$
above any codimension~$1$ point of $(\P^1_F \setminus \{0,\infty\})^n$ is split.
Assume, in addition, that for any rational point~$t$ of a dense open subset
of $(\P^1_F)^n$,
the set $Z_t(F)$ is dense in $Z_t(\A_F)^{\Br(Z_t)}$,
where $Z_t=f^{-1}(t)$.
Then $Z(F)$ is dense in $Z(\A_F)^{\Br(Z)}$.
\end{thm}

For $n=1$, this is
\cite[Th.~9.31]{hw}.
In view of \cite[Cor.~1.3]{ghs},
Theorem~\ref{app:th:arithmeticinput}
for any~$n$
follows from the $n=1$ case by a straightforward induction.

To prove Theorem~\ref{app:th:numberfield}, we
apply Theorem~\ref{app:th:arithmeticinput} to~$\phi'$, with $n=2$.
By our choice of~$X''$,
the fibres of~$\phi'$ above $(\P^1_F \setminus \{0,\infty\})^2$
are Weil restrictions of smooth projective conics.
In particular, the geometric generic fibre of~$\phi'$ is rational, hence rationally connected,
and the codimension~$1$ fibres of~$\phi'$
above $(\P^1_F \setminus \{0,\infty\})^2$
are smooth and geometrically irreducible, hence split.
To verify the arithmetic hypothesis on the closed fibres of~$\phi'$,
we note
that if~$Z_t$ is a Weil restriction, by a finite extension
of number fields,
of a smooth projective conic,
and if $Z_t(\A_F)\neq\emptyset$,
then
the Hasse--Minkowski theorem implies that $Z_t(F) \neq \emptyset$,
from which it follows that the variety~$Z_t$ is rational over~$F$
and hence that $Z_t(F)$ is dense in $Z_t(\A_F)$.
All in all, we conclude that~$X'(F)$ is dense in $X'(\A_F)^{\Br(X')}$.
Therefore $X(F)\neq\emptyset$ if and only if 
$X'(\A_F)^{\Br(X')} \cap \prod_{v \in \Omega}X(F_v) \neq \emptyset$.
By Theorem~\ref{app:th:anyfield},
the latter condition is implied by the one which appears in the statement
of Theorem~\ref{app:th:numberfield}, which, 
in view of the quadratic reciprocity law,
is itself implied by the existence of a rational point on~$X$.
Thus the proof of Theorem~\ref{app:th:numberfield} is complete.

\subsection{Brauer groups}

In the remainder of this appendix, we prove Theorem~\ref{app:th:anyfield}.
Hereafter~$F$ denotes a field of characteristic zero.
We start with two general remarks about the Brauer group of Weil restrictions of conics
and of trivial $2$\nobreakdash-dimensional tori.

\begin{prop}
\label{app:prop:brauerweilrestriction}
Let~$k'/k$ be a finite separable extension of fields.  Let~$C$ be a smooth, projective conic over~$k'$.
If $R_{k'/k}C$ denotes the Weil restriction of~$C$ from~$k'$ to~$k$,
the pull-back map $\Br(k)\to \Br(R_{k'/k}C)$ is surjective.
\end{prop}

\begin{proof}
Let~$\bark$ be a separable closure of~$k$.
The choice of a $(k'\otimes_k\bark)$\nobreakdash-point of~$C$ determines an isomorphism
between $(R_{k'/k}C) \otimes_k \bark$
and the product of $[k':k]$ copies of $\P^1_{\bark}$ indexed by the finite set $\Spec(k' \otimes_k \bark)$.
It follows that $\Br((R_{k'/k}C) \otimes_k \bark)=0$
and that $\mathrm{Pic}((R_{k'/k}C) \otimes_k \bark)$ is isomorphic,
as a Galois module, to $\Z^{\Spec(k' \otimes_k \bark)}$.
Hence, by Shapiro's lemma, the terms $E_2^{0,2}$ and $E_2^{1,1}$ of the
Hochschild--Serre
spectral sequence
$$E_2^{p,q}=\h^p(k,\h^q_\et((R_{k'/k}C) \otimes_k \bark,\Gm))\Rightarrow \h^{p+q}_\et(R_{k'/k}C,\Gm)$$
vanish (see \cite[Ch.~III, Prop.~4.9]{milneet}).
We conclude that $E_2^{2,0}=\Br(k)$ surjects onto $\h^2_\et(R_{k'/k}C,\Gm)=\Br(R_{k'/k}C)$.
\end{proof}

Given a field~$k$ of characteristic different from~$2$
(typically $k=F(\beta,\gamma)$ or $k=F(X)$)
and two elements
$x,y\in k^*$, we denote by $(x,y)$ the class, in $\Br(k)$, of the
corresponding quaternion algebra over~$k$.

\begin{prop}
\label{app:prop:brauertorus}
Any $2$\nobreakdash-torsion element
of $\Br((\P^1_F \setminus \{0,\infty\})^2)$
can be written as
\begin{align*}
(\beta,r)+(\gamma,s)+\varepsilon(\beta,\gamma)+\delta
\end{align*}
for some $r,s \in F^*$,
some $\varepsilon \in \{0,1\}$, and some $\delta \in \mathrm{Im}\mkern1mu\big(\Br(F)\to\Br(F(\beta,\gamma))\big)$.
\end{prop}

\begin{proof}
Let $D_1 = \A^1_F \times \{0\}$, $D_2 = \{0\} \times \A^1_F$,
$D_{12}=D_1\cap D_2$, and $D_i^0=D_i \setminus D_{12}$.
By purity for \'etale cohomology,
we have $\h^q_{\et,D_{12}}(\A^2_F,\Z/2\Z)=\h^{q-4}_\et(D_{12},\Z/2\Z)=\h^{q-4}(F,\Z/2\Z)$ for any~$q$,
and $\h^3_{\et,D_1^0 \cup D_2^0}(\A^2_F \setminus D_{12},\Z/2\Z)=\h^1_{\et}(D_1^0,\Z/2\Z)
\oplus \h^1_{\et}(D_2^0,\Z/2\Z)$
(see \cite[Ch.~VI, \textsection5]{milneet}).
The long exact sequence of a triple
therefore yields an exact sequence
\begin{align*}
0\to \h^3_{\et,D_1 \cup D_2}(\A^2_F,\Z/2\Z) \to \h^1_{\et}(D_1^0,\Z/2\Z) \oplus \h^1_{\et}(D_2^0,\Z/2\Z) \to \Z/2\Z\rlap{,}
\end{align*}
where the rightmost map is the sum of the residues at~$0$
(\emph{op.\ cit.}, Ch.~III, Rem.~1.26).
Finally, let us consider the localisation exact sequence
\begin{align*}
\h^2_\et(\A^2_F,\Z/2\Z) \to \h^2_\et(\A^2_F \setminus (D_1\cup D_2),\Z/2\Z)
\to \h^3_{\et,D_1 \cup D_2}(\A^2_F,\Z/2\Z)
\end{align*}
(\emph{loc.\ cit.}, Prop.~1.25).
As any $2$\nobreakdash-torsion element of
$\Br((\P^1_F \setminus \{0,\infty\})^2)$
can be lifted to
$\h^2_\et(\A^2_F \setminus (D_1\cup D_2),\Z/2\Z)$
and as $H^1(D_i^0,\Z/2\Z) \subset H^1(F(D_i),\Z/2\Z)=F(D_i)^*/F(D_i)^{*2}$
is generated by $F^*/F^{*2}$ and by the class of~$\beta$ (resp.~$\gamma$) if $i=1$
(resp.~$i=2$),
we deduce from these two exact sequences that for
an arbitrary
$2$\nobreakdash-torsion
class $\alpha \in \Br((\P^1_F \setminus \{0,\infty\})^2)$,
the residues of~$\alpha$ at the generic points of~$D_1$ and~$D_2$,
viewed as elements of $F(D_i)^*/F(D_i)^{*2}$,
are represented
by elements of~$F(D_i)^*$
of the shape
$\beta^\varepsilon s$ and $\gamma^\varepsilon r$,
respectively,
for some $r,s\in F^*$ and some $\varepsilon \in \{0,1\}$.
The class $(\beta,r)+(\gamma,s)+\varepsilon(\beta,\gamma)$ has the same residues
as~$\alpha$ at these two points.
As $\h^2_\et(\A^2_F,\Z/2\Z)=\h^2(F,\Z/2\Z)$
(\emph{op.\ cit.}, Ch.~VI, Cor.~4.20),
these two classes differ by a constant class, in view of the localisation exact sequence.
\end{proof}

The following is a simple consequence of
Proposition~\ref{app:prop:brauerweilrestriction}
and Proposition~\ref{app:prop:brauertorus}.

\begin{prop}
\label{app:prop:explicitgenerators}
The cokernel of the natural map $\Br(F) \to \Br_\nr(F(X)/F)$
is killed by~$2$.  Its elements
are represented by classes of the shape
$$(\beta,r)+(\gamma,s)+\varepsilon(\beta,\gamma)$$
for $r,s \in F^*$ and $\varepsilon\in \{0,1\}$.
\end{prop}

\begin{proof}
The generic fibre of $\phi':X'\to \P^1_F \times \P^1_F$
is a Weil restriction of a smooth projective conic.
Applying Proposition~\ref{app:prop:brauerweilrestriction} to it,
we see
that any element of $\Br_{\nr}(F(X)/F)=\Br(X')$
can be written as $\phi'^*\alpha$
for some $\alpha \in \Br(F(\P^1_F \times \P^1_F))$.
As~$\phi'^*\alpha$ is unramified on~$X'$
and as the fibres of~$\phi'$ above $(\P^1_F \setminus \{0,\infty\})^2$ are split,
the class~$\alpha$ belongs to the subgroup
$\Br((\P^1_F \setminus \{0,\infty\})^2) \subset \Br(F(\P^1_F \times \P^1_F))$
(see \cite[Prop.~1.1.1]{ctsd94}).
After adding a constant class to~$\alpha$,
we may assume that the value of~$\alpha$ at the point~$(1,1)$
vanishes in~$\Br(F)$.
The next lemma
then implies that $2\alpha=0$
(\emph{op.\ cit.}, Prop.~1.3.3).
Applying Proposition~\ref{app:prop:brauertorus} now concludes the proof
of Proposition~\ref{app:prop:explicitgenerators}.
\end{proof}

\begin{lem}
\label{app:lem:residueskilledbytwo}
Let $i\in\{1,2\}$, $t\in\{0,\infty\}$.
Let $\xi \in\P^1_F \times \P^1_F$
be the generic point of the fibre, above~$t$, of the $i$th
projection
$\P^1_F \times \P^1_F \to \P^1_F$.
The residue of~$\alpha$ at $\xi$ is killed by~$2$.
\end{lem}

\begin{proof}
Let $K=F(\P^1_F \times \P^1_F) = F(\beta,\gamma)$
and $\sO_K \subset K$ denote the local ring of $\P^1_F\times \P^1_F$
at~$\xi$.
Let $K'/K$ be the quadratic extension obtained by adjoining
a square root of~$\beta$ if $i=1$, or of~$\gamma$ if $i=2$.
Let $\sO_{K'}$ denote the integral closure of~$\sO_K$ in~$K'$.
To prove the lemma, it suffices to check that the image of $\alpha \in \Br(K)$
in~$\Br(K')$ belongs to the subgroup $\Br(\sO_{K'}) \subset \Br(K')$
(see \cite[Prop.~1.1.2]{ctsd94}).
For this, as $\phi'^*\alpha$ is unramified over~$X'$, it suffices
to check that the $K'$\nobreakdash-variety $X' \times_{\P^1_F \times \P^1_F} \Spec(K')$ admits
a proper regular model over~$\sO_{K'}$ whose special fibre is split
(\emph{loc.\ cit.}, Prop.~1.1.1; this property does not depend on the choice
of the model \cite[Cor.~1.2]{skorodescent}).
A look at the equations which define~$X''$
shows that
$X' \times_{\P^1_F \times \P^1_F} \Spec(K')$
even has good reduction
over~$\sO_{K'}$
as it descends to a variety over $F(\gamma) \subset \sO_{K'}$ if $i=1$,
or over $F(\beta) \subset \sO_{K'}$ if $i=2$.
\end{proof}

To exploit the hypothesis that the classes under consideration are unramified,
we shall repeatedly use the following tool.

\begin{prop}
\label{app:prop:tool}
Let $r,s \in F^*$ and $\varepsilon\in \{0,1\}$
be such that
$(\beta,r)+(\gamma,s)+\varepsilon(\beta,\gamma) \in \Br_\nr(F(X)/F)$.
Let $w\in F^*$.
\begin{enumerate}
\item
If $\phi^{-1}(\nu_{\gamma=w})$ contains a rational point, then $rw^\varepsilon$
is a square in~$F$.
\item
If $\phi^{-1}(\nu_{\beta=w\gamma})$ contains a rational point, then $rs(-w)^\varepsilon$
is a square in~$F$.
\end{enumerate}
\end{prop}

\begin{proof}
Let $\nu \in \{\nu_{\gamma=w},\nu_{\beta=w\gamma}\}$.
Let~$\bar\nu$ denote the Zariski closure of~$\nu$ in $\P^1_F \times \P^1_F$.
Suppose that $\phi^{-1}(\nu)$ contains a rational point.
The inclusion $\bar\nu \hookrightarrow \P^1_F \times \P^1_F$ then factors through~$\phi'$.
As $(\beta,r)+(\gamma,s)+\varepsilon(\beta,\gamma) \in \Br((\P^1_F\setminus\{0,\infty\})^2)$ becomes, by assumption, unramified over~$X'$ when pulled back to~$X''$,
its value at~$\nu$
must therefore belong to the subgroup $\Br(\bar\nu) \subset \Br(\nu)$.
Letting~$t$
denote the coordinate of~$\P^1_F$,
this means that
when $\nu=\nu_{\gamma=w}$ (resp., $\nu=\nu_{\beta=w\gamma}$),
the class 
$(t,r)+(w,s)+\varepsilon(t,w)$
(resp., $(wt,r)+(t,s)+\varepsilon(wt,t)$)
in $\Br(\P^1_F \setminus \{0,\infty\})$
must be unramified over~$\P^1_F$;
hence the proposition.
\end{proof}

The next three propositions complete the proof of Theorem~\ref{app:th:anyfield}.

\begin{prop}
\label{app:prop:alternative}
Let $r,s \in F^*$.
If the class
$(\beta,r)+(\gamma,s)+(\beta,\gamma) \in \Br(F(X))$
belongs to $\Br_{\nr}(F(X)/F)$,
then it belongs
to the image of the natural map $\Br(F)\to \Br(F(X))$.
\end{prop}

\begin{proof}
We proceed in several steps.

Step 1: we claim that~$ad$ is a square in~$F$.

Suppose that~$ad$ is not a square in~$F$.  Then~$a$ and~$d$ are not both squares;
by symmetry, we may assume that~$a$ is not a square in~$F$;
after replacing~$F$ by~$F(\sqrt{d})$, we may then assume
that~$d$ is a square in~$F$.
As~$a$ is not a square in~$F$,
it cannot be a square in~$F(\sqrt{c}\mkern.7mu)$ and in~$F(\sqrt{ac}\mkern.7mu)$ at the same time.
After replacing~$F$ by one of these two extensions, we may therefore assume
that~$c$ or~$ac$ is a square.
The hypotheses of Proposition~\ref{app:prop:nulifts}~(1) are now met.
By Proposition~\ref{app:prop:tool}~(1) applied twice,
we deduce that~$rw$ and~$raw$ are squares in~$F$, hence~$a$ is a square in~$F$,
which is absurd.

Step 2: we claim that at least one of~$d$ and~$cd$ is a square in~$F$.

We may assume, after replacing~$F$ with~$F(\sqrt{c}\mkern.7mu)$,
that~$c$ is a square in~$F$.
Applying Proposition~\ref{app:prop:nulifts}~(1)
and Proposition~\ref{app:prop:tool}~(1),
we deduce that~$rw$ and~$rdw$ are squares in~$F$; hence~$d$ is a square in~$F$.

Step 3: we claim that at least one of~$a$ and~$ab$ is a square in~$F$.

This follows from Step~2, by symmetry.

Step 4: if $a=b=c=d$ in $F^*/F^{*2}$, then $a=b=c=d=1$ in $F^*/F^{*2}$.

Applying Proposition~\ref{app:prop:nulifts}~(2)
and Proposition~\ref{app:prop:tool}~(2),
we see that~$-rsw$ and $-rsaw$ are squares in~$F$, hence~$a$ is a square in~$F$.

Putting Steps~1 to~4 together, we have now proved that~$a$ and~$d$ are squares in~$F$.
Let us write $B=(B_1,B_2)$ and $C=(C_1,C_2)$ according to the decompositions
$F_a=F\times F$ and $F_d = F \times F$ induced by the choice of square roots of~$a$
and~$d$.

Step 5: we claim that $s=B_i$ and $r=C_j$ in $F^*/F^{*2}$ for some $i,j\in \{1,2\}$.

By symmetry, it suffices to check that $r=C_j$ in $F^*/F^{*2}$
for some $j\in\{1,2\}$.  To this end, 
as $C_1C_2=c$,
we may replace~$F$ with~$F(\sqrt{c}\mkern.7mu)$
and assume that~$c$ is a square in~$F$.
In this case, Proposition~\ref{app:prop:nulifts}~(3)
and Proposition~\ref{app:prop:tool}~(1)
together imply the claim.

As the equality $(\beta B_i,\gamma C_j)=0$ holds in $\Br(F(X))$ by the very definition of~$X$,
Step~5 implies that $(\beta,r)+(\gamma,s)+(\beta,\gamma)=(B_i,C_j)$,
which does come from $\Br(F)$.
\end{proof}

\begin{prop}
\label{app:prop:unramifiedrs}
Let $r,s \in F^*$ be such that
$(\beta,r)+(\gamma,s) \in \Br_{\nr}(F(X)/F)$.
Then at least one of the following holds:
\begin{enumerate}
\item $r$ and~$s$ are squares in~$F$;
\item $rc$ and~$s$ are squares in~$F$, at least one of $a$, $c$, $ad$, $cd$ is a square in~$F$,
and at least one of $a$, $b$, $c$, $d$, $abc$ is a square in~$F$;
\item $r$ and~$sb$ are squares in~$F$, at least one of $b$, $d$, $ad$, $ab$ is a square in~$F$,
and at least one of $a$, $b$, $c$, $d$, $bcd$ is a square in~$F$;
\item $rc$ and~$sb$ are squares in~$F$, at least one of $a$, $c$, $ad$, $cd$ is a square in~$F$,
and at least one of~$b$, $d$, $ad$, $ab$ is a square in~$F$.
\end{enumerate}
\end{prop}

\begin{proof}
Extending the scalars from~$F$ to~$F(\sqrt{c}\mkern.7mu)$
and applying
Proposition~\ref{app:prop:nulifts}~(1)
and Proposition~\ref{app:prop:tool}~(1)
shows that~$r$ is a square in~$F(\sqrt{c}\mkern.7mu)$,
hence at least one of~$r$ and~$rc$ is a square in~$F$.  By symmetry, at least
one of~$s$ and~$sb$ is a square in~$F$.
The following two assertions and the symmetric assertions will now imply the proposition:
\begin{enumerate}
\item[(i)] if~$rc$ is a square in~$F$, then at least one of~$a$, $c$, $ad$, $cd$ is a square in~$F$;
\item[(ii)] if~$rc$ and~$s$ are squares in~$F$, then at least one of $a$, $b$, $c$, $d$, $abc$ is a square in~$F$.
\end{enumerate}
To prove~(i), we may replace~$F$ with~$F(\sqrt{ac},\sqrt{d})$
and assume that~$ac$ and~$d$ are squares in~$F$.
Proposition~\ref{app:prop:nulifts}~(1)
and Proposition~\ref{app:prop:tool}~(1)
then imply that~$r$ is a square in~$F$; if~$rc$ is a square in~$F$,
it follows that~$c$ is a square in~$F$, as desired.
To prove~(ii), we may assume, in view of~(i), that~$ad$ or~$cd$ is a square in~$F$.
We may then replace~$F$ with $F(\sqrt{ab},\sqrt{ac}\mkern.7mu)$
and assume that~$ab$ and~$ac$ are squares in~$F$.
In this case,
Proposition~\ref{app:prop:nulifts}~(2)
and Proposition~\ref{app:prop:tool}~(2) imply that~$c$ is a square in~$F$.
\end{proof}

\begin{prop}
If~$a$ or~$c$ is a square in~$F$,
then $(\beta,c)\in\mathrm{Im}\mkern1mu(\Br(F)\to\Br(F(X)))$.
If~$b$ or~$d$ is a square in~$F$,
then $(\gamma,b)\in\mathrm{Im}\mkern1mu(\Br(F)\to\Br(F(X)))$.
\end{prop}

\begin{proof}
By symmetry, we need only check the first assertion.
If~$c$ is a square, it is trivial.
Let us assume that~$a$ is a square in~$F$ and write $B=(B_1,B_2)$
according to the decomposition
$F_a=F\times F$
induced by the choice of a square root of~$a$.
The vanishing of the class $(\beta B,\gamma C) \in \Br(F(X) \otimes_F F_{a,d})$,
which holds by the very definition of~$X$, is then equivalent to that of
the two classes $(\beta B_i,\gamma C) \in \Br(F(X) \otimes_F F_d)$, $i\in \{1,2\}$.
Taking the norm down to~$F(X)$ and applying the projection formula,
we deduce that $(\beta B_i,c)=0$ in~$\Br(F(X))$ for $i \in \{1,2\}$.
Hence $(\beta,c)=(B_i,c)$, which does come from $\Br(F)$.
\end{proof}

\begin{rmk}
\label{app:rmk:unram-iff}
Proposition~\ref{app:prop:unramifiedrs}
provides
necessary conditions for the classes $(\beta,c)$, $(\gamma,b)$ and $(\beta,c)+(\gamma,b)$
to belong to $\Br_\nr(F(X)/F)$.
It is possible to show, although we do not do it here, 
that these necessary conditions are in fact necessary and sufficient.
\end{rmk}

\begin{ex}
\label{app:example}
Let $F=\Q$, $a=d=34$, $b=2$, $c=17$, $B=6+\sqrt{34}$, $C=(17+2\sqrt{34})/3$.
It is easy to see that~$X$ has points everywhere locally.
Using the fact that at every place of~$\Q$, at least one of~$2$, $17$, and~$34$
is a square,
one can check that for any place~$v$ of~$\Q$ other than~$17$
(resp., for $v=17$),
if $\beta_v \in F_v^*$ denotes the~$\beta$ coordinate of any $\Q_v$\nobreakdash-point
of~$X$,
the Hilbert symbol $(\beta_v,17)$ is trivial (resp., is nontrivial).
Hence $X(\Q)=\emptyset$.
Thus, in this case, the three classes $(a,b),(b,c),(c,d) \in \Br(\Q)$ vanish, but the Massey product $\langle a,b,c,d \rangle$ is not defined, by Theorem~\ref{thm-XF-splitting-variety} and Theorem~\ref{thm-main-number-fields}.
By Theorem \ref{thm-main-number-fields-generic}, for such an example to exist, it is necessary that at least one of $ad$, $ab$, $cd$ is a square.
\end{ex}

% Fakesection bib
%% bibliography

\newcommand{\noopsort}[1]{} \newcommand{\printfirst}[2]{#1}
  \newcommand{\singleletter}[1]{#1} \newcommand{\switchargs}[2]{#2#1}
  \def\cftil#1{\ifmmode\setbox7\hbox{$\accent"5E#1$}\else
  \setbox7\hbox{\accent"5E#1}\penalty 10000\relax\fi\raise 1\ht7
  \hbox{\lower1.15ex\hbox to 1\wd7{\hss\accent"7E\hss}}\penalty 10000
  \hskip-1\wd7\penalty 10000\box7}


\begin{thebibliography}{10}

\bibitem{arason}
J.~K. Arason.
\newblock Cohomologische {I}nvarianten quadratischer {F}ormen.
\newblock {\em J. Algebra}, 36(3):448--491, 1975.

\bibitem{auslandergoldman}
M.~Auslander and O.~Goldman.
\newblock The {B}rauer group of a commutative ring.
\newblock {\em Trans. Amer. Math. Soc.}, 97:367--409, 1960.

\bibitem{cttoulouse}
J.-L. Colliot-Th{\'e}l{\`e}ne.
\newblock L'arithm\'etique des vari\'et\'es rationnelles.
\newblock {\em Ann. Fac. Sci. Toulouse Math. (6)}, 1(3):295--336, 1992.

\bibitem{ctfibres}
J.-L. Colliot-Th{\'e}l{\`e}ne.
\newblock Fibre sp\'eciale des hypersurfaces de petit degr\'e.
\newblock {\em C. R. Math. Acad. Sci. Paris}, 346(1-2):63--65, 2008.

\bibitem{ctsanmumbai}
J.-L. Colliot-Th{\'e}l{\`e}ne and J.-J. Sansuc.
\newblock The rationality problem for fields of invariants under linear
  algebraic groups (with special regards to the {B}rauer group).
\newblock In {\em Algebraic groups and homogeneous spaces}, Tata Inst. Fund.
  Res. Stud. Math., pages 113--186. Tata Inst. Fund. Res., Mumbai, 2007.

\bibitem{ctsd94}
J.-L. Colliot-Th{\'e}l{\`e}ne and P. Swinnerton-Dyer.
\newblock Hasse principle and weak approximation for pencils of
  {S}everi-{B}rauer and similar varieties.
\newblock {\em J.~reine angew. Math.}, 453:49--112, 1994.

\bibitem{MR0382702}
P.~Deligne, P.~Griffiths, J.~Morgan, and D.~Sullivan.
\newblock Real homotopy theory of {K}\"ahler manifolds.
\newblock {\em Invent. Math.}, 29(3):245--274, 1975.

\bibitem{MR0385851}
W.~G. Dwyer.
\newblock Homology, {M}assey products and maps between groups.
\newblock {\em J. Pure Appl. Algebra}, 6(2):177--190, 1975.

\bibitem{MR3239143}
I.~Efrat.
\newblock The {Z}assenhaus filtration, {M}assey products, and representations
  of profinite groups.
\newblock {\em Adv. Math.}, 263:389--411, 2014.

\bibitem{EM-triple}
I.~Efrat and E.~Matzri.
\newblock Triple {M}assey products and absolute {G}alois groups.
\newblock {\em J. Eur. Math. Soc. (JEMS)}, 19(12):3629--3640, 2017.
%\newblock to appear, arXiv:1412.7265.

\bibitem{MR3415664}
I.~Efrat and E.~Matzri.
\newblock Vanishing of {M}assey products and {B}rauer groups.
\newblock {\em Canad. Math. Bull.}, 58(4):730--740, 2015.

\bibitem{MR2863369}
I.~Efrat and J.~Min{\'a}{\v{c}}.
\newblock On the descending central sequence of absolute {G}alois groups.
\newblock {\em Amer. J. Math.}, 133(6):1503--1532, 2011.

\bibitem{elmanlam}
R.~Elman and T.~Y. Lam.
\newblock Quadratic forms under algebraic extensions.
\newblock {\em Math. Ann.}, 219(1):21--42, 1976.

\bibitem{fenn}
R.~A. Fenn.
\newblock {\em Techniques of geometric topology}, volume~57 of {\em London
  Mathematical Society Lecture Note Series}.
\newblock Cambridge University Press, Cambridge, 1983.

\bibitem{MR2018550}
W.~Gao, D.~B. Leep, J.~Min{\'a}{\v{c}}, and T.~L. Smith.
\newblock {G}alois groups over nonrigid fields.
\newblock In {\em Valuation theory and its applications, {V}ol. {II}
  ({S}askatoon, {SK}, 1999)}, volume~33 of {\em Fields Inst. Commun.}, pages
  61--77. Amer. Math. Soc., Providence, RI, 2003.

\bibitem{ghs}
T.~Graber, J.~Harris, and J.~Starr.
\newblock Families of rationally connected varieties.
\newblock {\em J. Amer. Math. Soc.}, 16(1):57--67, 2003.

\bibitem{grbr}
A.~Grothendieck.
\newblock Le groupe de {B}rauer, {I}, {II}, {III}.
\newblock In {\em Dix expos\'es sur la cohomologie des sch\'emas}, pages
  46--188. North-Holland, Amsterdam, 1968.

\bibitem{MR2597737}
C.~Haesemeyer and C.~Weibel.
\newblock Norm varieties and the chain lemma (after {M}arkus {R}ost).
\newblock In {\em Algebraic topology}, volume~4 of {\em Abel Symp.}, pages
  95--130. Springer, Berlin, 2009.

\bibitem{hw}
Y.~Harpaz and O.~Wittenberg.
\newblock On the fibration method for zero-cycles and rational points.
\newblock {\em Ann. of Math. (2)}, 183(1):229--295, 2016.

\bibitem{hopkins}
M.~J. Hopkins and K.~G. Wickelgren.
\newblock Splitting varieties for triple {M}assey products.
\newblock {\em J. Pure Appl. Algebra}, 219(5):1304--1319, 2015.

\bibitem{MR3314129}
D.~C. Isaksen.
\newblock When is a fourfold {M}assey product defined?
\newblock {\em Proc. Amer. Math. Soc.}, 143(5):2235--2239, 2015.

\bibitem{MR0098366}
W.~S. Massey.
\newblock Some higher order cohomology operations.
\newblock In {\em Symposium internacional de topolog\'\i a algebraica
  {I}nternational symposium on algebraic topology}, pages 145--154. Universidad
  Nacional Aut\'onoma de M\'exico and UNESCO, Mexico City, 1958.

\bibitem{Matzri}
E.~Matzri.
\newblock Triple {M}assey products in {G}alois cohomology.
\newblock {\em Manuscript}.
\newblock arXiv:1411.4146.

\bibitem{milneet}
J.~S. Milne.
\newblock {\em \'{E}tale cohomology}, volume~33 of {\em Princeton Mathematical
  Series}.
\newblock Princeton University Press, Princeton, N.J., 1980.

\bibitem{MR1405942}
J.~Min{\'a}{\v{c}} and M.~Spira.
\newblock Witt rings and {G}alois groups.
\newblock {\em Ann. of Math. (2)}, 144(1):35--60, 1996.

\bibitem{MT3}
J.~Min{\'a}{\v{c}} and N.~D. T{\^a}n.
\newblock Construction of unipotent {G}alois extensions and {M}assey products.
\newblock {\em Adv. Math.} 304:1021--1054, 2017.
%\newblock arXiv:1501.01346.

\bibitem{MTlocal}
J.~Min{\'a}{\v{c}} and N.~D. T{\^a}n.
\newblock Counting {G}alois $\mathbb{U}_4(\mathbb{F}_p)$-extensions using
  {M}assey products.
\newblock {\em J. Number Theory} 176:76--112, 2017.
%\newblock arXiv:1408.2586.

\bibitem{MT-kernel-unip}
J.~Min{\'a}{\v{c}} and N.~D. T{\^a}n.
\newblock The kernel unipotent conjecture and the vanishing of {M}assey products for odd rigid fields.
\newblock {\em Adv. Math.} 273:242--270, 2015.

\bibitem{MT}
J.~Min{\'a}{\v{c}} and N.~D. T{\^a}n.
\newblock Triple {M}assey products and {G}alois theory.
\newblock {\em J. Eur. Math. Soc.} 19(1):255--284, 2017.
%\newblock arXiv:1307.6624.

\bibitem{MT2}
J.~Min{\'a}{\v{c}} and N.~D. T{\^a}n.
\newblock Triple {M}assey products vanish over all fields.
\newblock {\em J. London Math. Soc.} 94(3):909--932, 2016.
%\newblock arXiv:1412.7611.

\bibitem{MR3452187}
J.~Min{\'a}{\v{c}} and N.~D. T{\^a}n.
\newblock Triple {M}assey products over global fields.
\newblock {\em Doc. Math.}, 20:1467--1480, 2015.

\bibitem{MR516917}
J.~W. Morgan.
\newblock The algebraic topology of smooth algebraic varieties.
\newblock {\em Inst. Hautes \'Etudes Sci. Publ. Math.}, (48):137--204, 1978.

\bibitem{MR876163}
J.~W. Morgan.
\newblock Correction to: ``{T}he algebraic topology of smooth algebraic
  varieties'' [{I}nst.\ {H}autes \'{E}tudes {S}ci.\ {P}ubl.\ {M}ath.\ {N}o.\ 48
  (1978), 137--204; {MR}0516917 (80e:55020)].
\newblock {\em Inst. Hautes \'Etudes Sci. Publ. Math.}, (64):185, 1986.

\bibitem{MR1925911}
M.~Morishita.
\newblock On certain analogies between knots and primes.
\newblock {\em J. Reine Angew. Math.}, 550:141--167, 2002.

\bibitem{MR2004124}
M.~Morishita.
\newblock Milnor invariants and {M}assey products for prime numbers.
\newblock {\em Compos. Math.}, 140(1):69--83, 2004.

\bibitem{neukirchant}
  J.~Neukirch.
  \newblock {\em Algebraic Number Theory}, volume 322 of {\em Grundlehren der
  Mathematischen Wissenschaften [Fundamental Principles of Mathematical
    Sciences]}.
  \newblock Springer-Verlag, Berlin, 1999.

\bibitem{MR2392026}
J.~Neukirch, A.~Schmidt, and K.~Wingberg.
\newblock {\em Cohomology of number fields}, volume 323 of {\em Grundlehren der
  Mathematischen Wissenschaften [Fundamental Principles of Mathematical
  Sciences]}.
\newblock Springer-Verlag, Berlin, second edition, 2008.

\bibitem{Rost-chain-lemma}
M.~{Rost}.
\newblock Chain lemma for splitting fields of symbols.
\newblock {\em Preprint}, 1998.

\bibitem{saltmanbrauer}
D.~J. Saltman.
\newblock The {B}rauer group and the center of generic matrices.
\newblock {\em J. Algebra}, 97(1):53--67, 1985.

\bibitem{serrecg}
J.-P. Serre.
\newblock {\em Cohomologie galoisienne}, volume~5 of {\em Lecture Notes in
  Mathematics}.
\newblock Springer-Verlag, Berlin, fifth edition, 1994.

\bibitem{MR2312552}
R.~T. Sharifi.
\newblock {M}assey products and ideal class groups.
\newblock {\em J. Reine Angew. Math.}, 603:1--33, 2007.

\bibitem{skorodescent}
A.~N. Skorobogatov.
\newblock Descent on fibrations over the projective line.
\newblock {\em Amer. J. Math.}, 118(5):905--923, 1996.

\bibitem{skorodescenttoric}
A.~N. Skorobogatov.
\newblock Descent on toric fibrations.
\newblock In {\em Arithmetic and geometry}, volume 420 of {\em London Math.
  Soc. Lecture Note Ser.}, pages 422--435. Cambridge Univ. Press, Cambridge,
  2015.

\bibitem{MR0646078}
D.~Sullivan.
\newblock Infinitesimal computations in topology.
\newblock {\em Inst. Hautes \'Etudes Sci. Publ. Math.}, (47):269--331 (1978),
  1977.

\bibitem{MR2220090}
A.~Suslin and S.~Joukhovitski.
\newblock Norm varieties.
\newblock {\em J. Pure Appl. Algebra}, 206(1-2):245--276, 2006.

\bibitem{tate}
J.~Tate.
\newblock Relations between {$K\sb{2}$} and {G}alois cohomology.
\newblock {\em Invent. Math.}, 36:257--274, 1976.

\bibitem{MR2811603}
V.~Voevodsky.
\newblock On motivic cohomology with {$\bold Z/l$}-coefficients.
\newblock {\em Ann. of Math. (2)}, 174(1):401--438, 2011.

\bibitem{MR2132673}
D.~Vogel.
\newblock On the {G}alois group of 2-extensions with restricted ramification.
\newblock {\em J. Reine Angew. Math.}, 581:117--150, 2005.

\bibitem{zbMATH05594008}
C.~{Weibel}.
\newblock {The norm residue isomorphism theorem.}
\newblock {\em {J. Topol.}}, 2(2):346--372, 2009.

\bibitem{wickelgren2009lower}
K.~Wickelgren.
\newblock {\em Lower central series obstructions to homotopy sections of curves
  over number fields}.
\newblock PhD thesis, 2009.

\bibitem{MR3051256}
K.~Wickelgren.
\newblock {$n$}-nilpotent obstructions to {$\pi_1$} sections of {$\Bbb
  P^1-\{0,1,\infty\}$} and {M}assey products.
\newblock In {\em {G}alois-{T}eichm\"uller theory and arithmetic geometry},
  volume~63 of {\em Adv. Stud. Pure Math.}, pages 579--600. Math. Soc. Japan,
  Tokyo, 2012.

\bibitem{MR3220523}
K.~Wickelgren.
\newblock On 3-nilpotent obstructions to {$\pi_1$} sections for {$\Bbb
  {P}^1_\Bbb Q-\{0,1,\infty\}$}.
\newblock In {\em The arithmetic of fundamental groups---{PIA} 2010}, volume~2
  of {\em Contrib. Math. Comput. Sci.}, pages 281--328. Springer, Heidelberg,
  2012.

\end{thebibliography}
\end{document}